\documentclass[11pt]{amsart}
\usepackage{amsmath, amsthm, amssymb, color,xcolor}
\usepackage[ pdfborderstyle={/S/U/W 1}]{hyperref}
\definecolor{MyBlue}{HTML}{210cac}
\definecolor{MyCiteColor}{HTML}{0099FF}
\definecolor{MyRed}{HTML}{3E186A}
\definecolor{Red2}{HTML}{FF6E00}
\hypersetup{%
  pdfborderstyle={/S/U/W 1},
  colorlinks=true,
  linkcolor=MyBlue,
  citecolor=MyCiteColor,
  urlcolor=MyRed,
  anchorcolor=MyBlue,
  linkbordercolor=MyRed,
}
	\definecolor{oran}{HTML}{0A7F5E}
\usepackage[shortlabels]{enumitem}
\usepackage{makecell}
\usepackage{pgfplots}
\usepackage[final]{pdfpages}
\setboolean{@twoside}{false}
\usepackage{tabu}
\usepackage{pdfpages}
\usepackage{tikz}
\usepackage{chemfig}
\usepackage{subcaption}
\usetikzlibrary{shapes,arrows,spy,positioning,decorations}
\usepackage{caption}
\usepackage{subcaption}
\usepackage{mathrsfs,upgreek}
\usepackage{scalefnt}
\usepackage{verbatim}
\usepackage[version=3]{mhchem}
\usepackage{wrapfig}
\usepackage{cutwin}
\calclayout

\makeatletter
\newcommand{\eqnum}{\refstepcounter{equation}\textup{\tagform@{\theequation}}}
\newcommand\numberthis{\addtocounter{equation}{1}\tag{\theequation}}
\makeatother

\newtheorem{theorem}{Theorem}
\numberwithin{theorem}{section}
\newtheorem{proposition}[theorem]{Proposition}
\newtheorem{lemma}[theorem]{Lemma}

\newtheorem{conjecture}[theorem]{Conjecture}

\theoremstyle{definition}

\newtheorem{remark}[theorem]{Remark}
\newtheorem{example}[theorem]{Example}

\newcommand{\RR}{\mathbb{R}}
\newcommand{\R}{\mathbb{R}}

\newcommand{\PP}{\mathbb{P}}

\newcommand{\ZZ}{\mathbb{Z}}

 \date{}

\renewcommand{\k}{\upkappa}
\newcommand{\cc}{\mathsf{c}}

\DeclareMathOperator*{\sign}{sgn}
\DeclareMathOperator*{\im}{im}

\DeclareMathOperator*{\sgn}{sgn}

\DeclareMathOperator*{\tr}{Tr}

\title{{Multistationarity and Bistability} for Fewnomial Chemical Reaction Networks}
\author{Elisenda Feliu}
\address{Elisenda Feliu \\ Department of Mathematical Sciences \\
University of Copenhagen \\
Universitetsparken 5 \\
DK-2100 Copenhagen, Denmark}
\email{efeliu@math.ku.dk}
\author{Martin Helmer}
\address{Martin Helmer \\ Department of Mathematical Sciences \\
University of Copenhagen \\
Universitetsparken 5 \\
DK-2100 Copenhagen, Denmark}
\email{m.helmer@math.ku.dk }
\begin{document} 
\renewcommand\windowpagestuff{\rule{2cm}{2cm}}
\opencutleft

\begin{abstract} \noindent
Bistability and multistationarity are properties of reaction networks linked to switch-like responses and connected to cell memory and cell decision making.
Determining whether and when a network exhibits bistability is a hard and open mathematical problem. One successful strategy consists of analyzing small networks and deducing that some of the properties are preserved upon passage to the full network. 
Motivated by this we study chemical reaction networks with few chemical complexes. Under mass-action kinetics the steady states of these networks are described by fewnomial systems, that is polynomial systems having few distinct monomials. Such systems of polynomials are often studied in real algebraic geometry by the use of Gale dual systems.  Using this Gale duality we give precise conditions in terms of the reaction rate constants for the number and stability of the steady states of families of reaction networks with one non-flow reaction. 
\end{abstract}
\maketitle

\color{black}

\section{Introduction}

Bistability, that is, the existence of two asymptotically stable steady states together with an unstable steady state, is a key property of dynamical systems that provides an explanation of switch-like behavior in real systems. In particular, bistability, and more generally multistability, are linked to cell decision making and differentiation, explaining the coexistence of different states in cells with identical genetic material \cite{Ferrell:Waddington,laurent1999,Xiong:2003jt}.  
Although the term bistability is widespread in molecular biology, determining whether and when a mathematical model exhibits bistability is a highly nontrivial task. This task is further complicated by the high number of variables and unknown parameters present in the models. Since bistability requires multistationarity, that is, the existence of more than one steady state, much effort has been centered around the easier (but still complex) problem of determining whether and when a network displays multistationarity. 

One of the successful strategies to address multistationarity and bistability focuses on studying a smaller, but related, model, and then seeks to lift these properties to the larger full model. To effectively implement this strategy one requires a catalog of small networks that are multistationary or bistable. We contribute to this catalog through a detailed analysis of networks with one non-flow reaction. 

Specifically, 
we consider the mathematical framework in which 
\color{black}
the evolution of the concentration of species in a chemical reaction network is modeled by means of a system of polynomial ordinary differential equations (ODEs). This system has the form  
\[
\frac{dx}{dt} = f_{\kappa} (x), \qquad x\in \R^n_{\geq 0},
\]
where $x=(x_1,\dots,x_n)$ is the vector of concentrations, $\kappa\in \R^s_{>0}$ is the vector of reaction rate constants, and $f_{\kappa}(x)=[f_{\kappa,1}(x),\dots, f_{\kappa,n}(x)]^T \in (\R[x_1,\dots,x_n])^n$ is a vector of $n$ variable polynomials.
Due to $\R$-linear relations among the polynomials $f_{\kappa,i}(x)$, $i=1,\dots,n$, the dynamics of this ODE system are often confined to linear subspaces defined by relations
\begin{equation}\label{eq:conservation_laws}
 W x = c, 
\end{equation}
where $c\in \R^d$ is determined by the initial conditions and $W$ is a matrix in $\R^{d\times n}$. These equations are called conservation laws.

In this setting, the positive steady states of the network are  the elements of the set
$$ S_{\kappa,c} = \{ x\in \R^n_{>0} \mid f_\kappa(x)=0,\ Wx = c\}. $$
Multistationarity then refers to the existence of a choice of the parameters $\kappa$ and $c$ such that the set $S_{\kappa,c}$ contains at least two elements.
The question of deciding \textit{whether} a network exhibits multistationarity is essentially solved, and one can employ one of several existing methods, see for example \cite{PerezMillan,conradi-switch,crnttoolbox,feliu-bioinfo,control}. On the other hand, deciding for \emph{what} parameter values the network exhibits multistationarity, deciding \emph{how many} elements the set $S_{\kappa,c}$ can contain, and deciding whether bistability arises remain open (and hard) problems.

As already discussed, it can be fruitful to have a detailed understanding of smaller networks contained within the larger network. In particular, it has been shown that after introducing inflow/outflow reactions \cite{craciun-feinberg} or adding so-called intermediates to a reaction network \cite{feliu:intermediates}, the maximum number of elements in the sets $S_{\kappa,c}$ can only increase. Similar relations are found for subnetworks or embedded networks \cite{joshi-shiu-II}. 
Motivated by this, in \cite{Joshi} a characterization of multistationarity was given for generic reaction networks with one `arbitrary' reaction and inflow/outflow reactions for all species; however neither the actual number of positive steady states nor the existence of bistability were  determined.
 In this work we expand upon this characterization by counting the possible number of steady states when the non-flow reaction is irreversible and by exploring the parameter region with respect to the cardinality of $S_{\kappa,c}$. In particular, we determine that these networks have at most three positive steady states. Additionally we also determine which of these networks exhibit bistability.

An additional goal of this paper is to present a strategy that can enhance our understanding of multistationarity, when the number of complexes appearing in the network is small. 
In principle, using cylindrical algebraic decompositions one can determine the region of multistationarity and the number of elements of $S_{\kappa,c}$, see for example, the book \cite{basu2007algorithms}. In practice, however, computing cylindrical algebraic decompositions is often unfeasible due to the high computational cost in relation to the degree, number of variables, and number of parameters in the polynomial systems being studied. 
In light of this, case specific approaches (sometimes in conjunction with partial results from \cite{Conradi:PLOSCOMP}) are often employed.

Here we study the number of steady states by finding a \emph{Gale dual system} \cite{bihan2008gale,bihan2007new,sottile2011real}. Let $l>0$ be an integer and consider a square system of $n$ polynomial equations with $n$ variables, $n+l+1$ monomials, and having a finite number of solutions. A Gale dual system is a new system of $l$ equations in $l$ variables, together with a cone, such that the solutions of the new system in the cone are in one-to-one correspondence with the \emph{positive} solutions of the original system. Hence, it can be advantageous to pass to the Gale dual system when $l$ is small, i.e.~when the original system has few monomials. Such systems are often called \textit{fewnomial} systems.

Moving to the Gale dual system is particularly advantageous when $l=1$, as is the case for the networks in \cite{Joshi} with one non-flow irreversible  reaction. The Gale dual system in this case is a single polynomial in one variable and the constraining cone is simply an interval on the real line. In this setting classical methods such as the (generalized) Descartes' rule of signs, Sturm sequences, or real analytical techniques in $\R$, can be applied to study the solutions in terms of the coefficients (which correspond to the unknown reaction rate constants).   

Using Gale dual systems we also study a second family of networks from \cite{Joshi}, constructed similarly to the case above, but where the non-flow reaction is reversible. In this case $l=2$, but with an appropriate choice of a Gale dual system, the problem is again reduced to the study of the roots of a single variable polynomial in an interval. The resulting polynomial in this case has a more complicated structure. However, we determine the number of solutions in terms of the reaction rate constants in some special cases. Based on these results we conjecture that the maximum number of positive steady states is also three for this second family of networks. We now give an example which illustrates our approach. 

\begin{example} Consider the chemical reaction network 
$$ X_1+X_2 \ce{ ->[\ell]} 5X_1+17X_2, \qquad 0 \ce{<=>[k_1][k_2]} X_1,\qquad  0\ce{<=>[k_3][k_4]} X_2.   $$
In this case there are no conservation laws. Hence, setting $\k=[\ell,k_1,k_2,k_3,k_4]$, we have that
$S_{\k}  = \{ x\in \R^2_{>0} \mid 4 \ell x_1x_2- k_2x_1+k_1=0, \ 16\ell x_1x_2-k_4x_2+k_3=0  \}.$
The system has two variables and $n+l+1=2+1+1=4$ monomials. Hence a Gale dual system consists of a one variable polynomial and a real interval. 
A Gale dual system is 
$$ p(y)= 64 \ell^{2}y^2+(-k_2k_4+16 k_1\ell+4k_3 \ell)y+k_1 k_3,\qquad y>0. $$
The number of solutions to this system agrees with the number of elements in $S_\k$. Since $p(y)$ has degree $2$, there are at most two solutions. 
From a Sturm sequence for $p(y)$ we obtain the set of constraints which  $k_1,k_2,k_3,k_4,\ell \in \RR_{>0}$ must satisfy for $p(y)$ to have two positive roots:
$$
k_2k_4-16k_1\ell-4k_3\ell>0,\quad  k_2^2k_4^2-(32k_1+8k_3)k_2k_4\ell+(256k_1^2-128k_1k_3+16k_3^2)\ell^2>0.
$$
The parameters $\k= [1, 33602, 15447, 35984, 8034]$ satisfy the inequalities above. 
With this choice of parameters, {and with $\omega=\sqrt {423113764748297}$}, the positive steady states are:\footnotesize\small
$$
{S_\k =  \left\lbrace \left[\frac{20749149-\omega}{82384},\frac{20617917-\omega}{10712}\right], \left[\frac{20749149+\omega}{82384},\frac{20617917+\omega}{10712}\right]\right\rbrace\cong \{[2.2,4.5],[501.5,3845] \}.}
$$\normalsize
\end{example}
The paper is organized as follows. In \S\ref{section:background} we give a brief overview of the mathematical study of chemical reaction networks, recall the notion of embedded networks, and introduce Gale dual systems. In \S\ref{sec:OneVarGaleSys} we study the family of networks from \cite{Joshi} with one irreversible non-flow reaction. In \S\ref{subsec:numberSteady_1} we give criteria for the reaction network to have zero, one, two, or three positive steady states and show that there can be no more than three steady states. In \S\ref{subsec:StabSteady_1} we investigate the stability of the steady states discussed in \S\ref{subsec:numberSteady_1}. Finally, in \S\ref{sec:reversible}, we modify the family of reaction networks discussed in \S\ref{sec:OneVarGaleSys} to include a reversible reaction and obtain results regarding the possible number and stability of steady states in certain cases. 

\section{Background}\label{section:background}

We begin this section with a brief review of some elements of chemical reaction network theory. Following this we will give an overview of the construction of Gale dual systems.

\subsection{Reaction networks and embedded networks}
Informally, a reaction network is a collection of reactions between linear combinations of species in a set $\{X_1,\dots,X_n\}$:
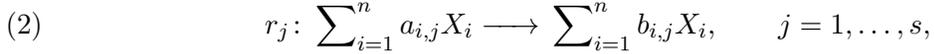
\begin{equation}\label{eq:reactionnetwork}
r_j\colon \sum\nolimits_{i=1}^n a_{i,j} X_i \ce{->} \sum\nolimits_{i=1}^n b_{i,j} X_i,  \qquad j=1,\dots,s,
\end{equation}
such that {$a_{i,j},b_{i,j}\in \ZZ_{\geq 0}$}. A \emph{reaction rate constant} $k_j>0$ is associated with each reaction, and is typically written as a label of the reaction. We let $x_i$ denote the concentration of $X_i$. The concentrations of the species over time are modeled by the following system of ODEs:
\begin{equation}\label{eq:ODE}
\frac{dx_i}{dt} =f_{k,i}(x),\quad \textrm{where }\quad f_{k,i}(x)=  \sum_{j=1}^s (b_{i,j} - a_{i,j} ) k_j \prod_{\ell=1}^n x_\ell^{a_{\ell, j}},\qquad i=1,\dots,n.  
\end{equation}
This ODE system arises from the assumption of mass-action kinetics. Throughout this paper, and without further reference, we will assume that \textit{all} reaction networks are equipped with mass-action kinetics.

The positive steady states of the network \eqref{eq:reactionnetwork} are the positive solutions $x\in \R^n_{>0}$ to the polynomial system $ f_{k,i}(x)=0$, for $i=1,\dots,n$. Given a network with reactions as in \eqref{eq:reactionnetwork}, an \emph{embedded network}
is obtained by removing a subset of species. Specifically, if we fix a subset $\mathcal{X}=\{X_{i_1},\dots,X_{i_\omega}\}$ of $\{X_1,\dots,X_n\}$, the embedded network on the subset $\mathcal{X}$ has reactions 
$$
 \sum\nolimits_{\nu=1}^\omega a_{i_\nu, j} X_{i_\nu} \ce{->} \sum\nolimits_{\nu=1}^\omega b_{i_\nu, j} X_{i_\nu},
$$
for all $j$ such that the two ends of the reaction are distinct. Repeated reactions are considered only once. For more on chemical reaction network theory see, for example, \cite{feinbergnotes,gunawardena-notes}.

In this paper we study reaction networks which contain inflow reactions $0\ce{->[\k_i]} X_i$ and outflow reactions $X_i\ce{->[\cc_i]} 0$, for all $i$. 
In particular, each equation in \eqref{eq:ODE} contains a linear term $-\cc_i x_i$ and an independent term $\k_i$. 
It follows that there are no $\R$-linear relations among the $ f_{k,i}(x)$ that hold for all $k\in \R^s$. Hence these networks have no conservation laws, see \eqref{eq:conservation_laws}. For a reaction network without conservation laws, a steady state $x^*$ is said to be \emph{non-degenerate} if the Jacobian matrix of $f_k=[f_{k,1},\dots,f_{k,n}]$ evaluated at $x^*$ is non-singular. 

For reaction networks with inflow and outflow reactions, we can apply the results in \cite{joshi-shiu-II} which relate the possible numbers of positive steady states of the given network with the possible numbers of steady states for the embedded networks. 
 We  state a simplified version of Theorem 4.2 in  \cite{joshi-shiu-II} for the case where the network has mass-action kinetics (the original theorem of \cite{joshi-shiu-II} is valid in greater generality). See the recent results in  \cite{Shiu-Wolff} for relaxing the assumption on non-degeneracy of the steady states.

\begin{proposition}[{\color{MyCiteColor} Theorem 4.2} \cite{joshi-shiu-II}]
\label{prop:embedded}
Let $N$ be a reaction network with inflow and outflow reactions for all its species, and let $N'$ be an embedded network.  If for some choice of reaction rate constants, $N'$ admits $L$ positive non-degenerate steady states, 
then there exist reaction rate constants such that $N$ also admits $L$ positive non-degenerate steady states. Further, if among the $L$ positive non-degenerate steady states of $N'$, $R$ are asymptotically stable and $L-R$ unstable, then this is also the case for $N$, for some choice of reaction rate constants.
\end{proposition}

Another result of \cite{joshi-shiu-II} which we employ in \S \ref{sec:reversible} concerns the process of making an irreversible reaction (i.e.~a reaction $y\rightarrow y'$ for which $y'\rightarrow y$ is not in the network) into a reversible reaction by adding the missing reverse reaction. Again we state a simplified version of the result here.

\begin{proposition}[{\color{MyCiteColor} Theorem 3.1} \cite{joshi-shiu-II}]
\label{prop:subnets}
Let $N'$ be a reaction network containing an irreversible reaction $y\rightarrow y'$ and let $N$ be the network obtained by adding $y'\rightarrow y$ to $N'$.  
If for some choice of reaction rate constants, $N'$ admits $L$ positive non-degenerate steady states, 
then there exist reaction rate constants such that $N$ also admits $L$ positive non-degenerate steady states.
Further, if among the $L$  positive non-degenerate steady states of $N'$,  $R$  are asymptotically stable and $L-R$ unstable, then this is also the case for $N$, for some choice of reaction rate constants.
\end{proposition}

\subsection{Gale Duality}\label{sec:galedual}

We now give a brief overview of the construction of Gale dual polynomial systems. Given a square system of polynomial equations this construction allows us to obtain a new \textit{Gale dual} system whose solutions are in bijective correspondence with the positive real solutions to the original system. We will see that in some cases the Gale dual system is simpler to study than the original system. For the interested reader, more details are given in Appendix \ref{app:galedual}, see also the book \cite{sottile2011real}.

Throughout this paper we will use standard multinomial notation, that is if we are working with (Laurent) polynomials in variables $x_1,\dots, x_n$ we will write the (Laurent) monomial $x_1^{w_1}\cdots x_{n}^{w_{n}}$ as $x^w$ for $w=[w_1,\dots, w_n]^T\in \ZZ^n$. Similarly, for a matrix $W\in \ZZ^{n\times m}$ with column vectors $w^{(1)},\dots,w^{(m)}$, we will write $x^W= [x^{w^{(1)}},\dots,x^{w^{(m)}}]^T$. To construct a Gale dual system we will compute Gale duals of matrices. A matrix $Q$ is \textit{Gale dual} to a matrix $A$ if the columns of $Q$ form a basis for the kernel of $A$, so that ${\rm Col}(Q)=\ker(A)$ and $A\cdot Q=0$. If $A$ is an integer matrix we, additionally, require that the Gale dual matrix $Q$ is also an integer matrix. 

In this subsection we study the positive real solutions to a polynomial system with $n$ equations and $n$ variables:
\begin{equation}\label{eq:compInt2}
f_1(x_1,\dots,x_{n})=\cdots =f_n(x_1,\dots,x_{n})=0, \;\;\; x_i\in \RR_{>0}.
\end{equation} 
In the notation defined above the system \eqref{eq:compInt2} can be written as 
\begin{equation}\label{eq:originalsystem}
 Cx^W= \begin{bmatrix}
f_1(x_1,\dots,x_{n}) \\
\vdots\\
f_n(x_1,\dots,x_{n})
\end{bmatrix}  =0
\end{equation} 
where $x^W$ consists of the $(n+l+1)$ unique monomials (possibly including the monomial $1$ if there are constant terms) which appear in the polynomials $f_1,\dots, f_n$ defining the system \eqref{eq:compInt2}. We adopt the convention that $C$ is an $n \times (n+l+1) $ matrix and that $W$ is an $n\times (n+l+1)$ integer matrix. Further we may assume, without loss of generality, that $1$ is the last monomial in the vector $x^W$. In particular, if this were not the case, since we only consider solutions with non-zero coordinates, we could simply divide the system \eqref{eq:originalsystem} by the given last monomial to obtain a system of the same form with $1$ as the last entry in the new vector $x^W$. 

Assume that $l> 0$ and that the system \eqref{eq:originalsystem} has a finite number of solutions. This implies that $C$ and $W$ both have maximal rank $n$.  
Let $Q=[q_{i,j}]$ be an $(n+l+1)\times (l+1)$ integer matrix Gale dual to $W$ chosen so that its last column is $[0,\dots,0,1]^T\in \ker(W)$ and let $D=[d_{i,j}]$ be an $(n+l+1)\times (l+1)$ matrix Gale dual to $C$ chosen to have its last row given by the vector $[0,0,\dots,1]$. Note that the choice of the last row of $D$ can be achieved by column operations. A vector  
 $x\in \RR^n_{>0}$ is a solution to \eqref{eq:originalsystem} if and only if the vector $x^W=[x^{w^{(1)}},\dots, x^{w^{(n+l)}},1]$ belongs to $\ker(C)$, or equivalently, if and only if there exists $y_1,\dots,y_{l}$ such that 
\begin{equation}\label{eq:newsystem} x^W = D \cdot [y_1,\dots,y_{l},1]^T .
\end{equation}
{For a column vector $ z\in \RR_{>0}^{n+l+1}$ we have that}
\small\begin{align*}
x^W=z \ \textrm{for some } x \in \RR_{>0}^{n}& \quad \Leftrightarrow \quad W^T \log(x)= \log(z) \ \textrm{ for some } x \in \RR_{>0}^{n} \quad \Leftrightarrow \quad  \log(z)\in \im(W^T)\\ &
\quad \Leftrightarrow \quad Q^T \log(z)=0  \quad \Leftrightarrow \quad  \log(z^{Q})=0 \quad \Leftrightarrow \quad  
z^{Q} = [1,1,\dots,1]^T.
\end{align*}\normalsize
{Since $W$ has full rank, $z$ and $x$ determine each other from the equation $x^W=z$. }
Using this  with $z=D\cdot [y_1,\dots,y_{l},1]^T$ in \eqref{eq:newsystem},  
 positive solutions to \eqref{eq:originalsystem} are in one-to-one correspondence with vectors $[y_1,\dots,y_{l}]\in \RR^{l}$ which satisfy the $l+1$ equations
 \begin{equation}\label{eq:galedual1}
\prod_{i=1}^{n+l+1} d_i(y)^{q_{i,1}} = 1 , \quad  \dots \quad  \prod_{i=1}^{n+l+1} d_i(y)^{q_{i,l+1}} = 1 ,\;\;\; {\rm such\; that}\; \; d_i(y)>0,
\end{equation}where the $d_i(y)$ are the linear forms in $\R[y_1,\dots,y_{l}]$ defined by the rows of $D\cdot [y_1,\dots,y_{l},1]^T$. That is,\begin{equation}\label{eq:di}
d_i(y):= (D\cdot [y_1,\dots,y_{l},1]^T  )_i = d_{i,l+1} + \sum_{r=1}^{l} d_{i,r} y_r, \;\;\; {\rm for\;}i=1,\dots, n+l+1.
\end{equation} Note that $d_{n+l+1}(y)=1$ by our choice of the last row of $D$ and that the last equation in \eqref{eq:galedual1} is simply $1=1$ by our choice of the last column of $Q$. Hence the solutions of the system \eqref{eq:originalsystem} are in one-to-one correspondence with solutions of the system of $l$ equations given by \begin{equation} \label{eq:galedual}
\prod_{i=1}^{n+l} d_i(y)^{q_{i,1}} = 1 ,\quad  \dots \quad  \prod_{i=1}^{n+l} d_i(y)^{q_{i,l}} = 1 ,\;\;\; {\rm such\; that}\; \; d_i(y)>0.
\end{equation} Additionally, this correspondence preserves the \emph{scheme} structure of the two systems, and in particular the multiplicity of the solutions is preserved. For details see Appendix \ref{app:galedual}. We summarize this in the following proposition.

\begin{proposition}\label{Prop:GaleDual2}
Let $l>0$, $w^{(i)}\in \ZZ^n$, and define the matrix $W=[{w^{(1)}},\dots, {w^{(n+l)}},0]$. Consider the system of $n$ (Laurent) polynomials in variables $x_1,\dots, x_n$ given by 
 $$ C\cdot x^W  =0,
$$
where $C\in \RR^{n\times (n+l+1)}$. Let $D\in \R^{{(n+l+1)}\times n}$ be a matrix with last row $[0,\dots,0,1]$ which is Gale dual to $C$ and let $Q=[q_{i,j}]\in \ZZ^{{(n+l+1)}\times n}$ be a matrix with last column $[0,\dots,0,1]^T$ which is Gale dual to $W$. There is a one-to-one, {multiplicity preserving}, correspondence between the set of solutions $x\in \RR_{>0}^{n}$ to $C \cdot x^W=0$ and the set of solutions to the system of $l$ equations
$$ \prod_{i=1}^{n+l} d_i(y)^{q_{i,1}} = 1,\quad  \dots \quad  \prod_{i=1}^{n+l} d_i(y)^{q_{i,l}} = 1 ,\;\;\; {\rm such\; that}\; \; d_i(y)>0 $$
in the variables $y_1,\dots,y_{l}$, where $d_i(y)$ is as in \eqref{eq:di}.
\end{proposition}  

We apply Gale duality to systems of polynomial equations defining the steady states of a network. The fact that the Gale dual system preserves the multiplicity of solutions implies that the solutions to the Gale dual system of multiplicity one correspond to non-degenerate steady states.

\section{Networks with one non-flow irreversible reaction}\label{sec:OneVarGaleSys}
In this section we apply Gale duality to study the number of positive steady states of a particular family of reaction networks, originally introduced in \cite{Joshi} (see also \cite{joshi-shiu-II, joshi-shiu-IV,joshi-shiu-V}). Specifically, consider a reaction network in $n$ species $X_1,\dots,X_n$ of the form 
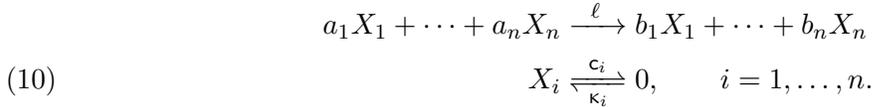
\begin{align*}
 a_1X_1+\cdots +a_nX_n &  \ce{->[\ell]} b_1X_1+\cdots+b_nX_n \\ 
X_i &\ce{<=>[\cc_{i}][\k_{i}]} 0,\qquad i=1,\dots,n. \numberthis \label{eq:myreaction}
\end{align*}
Each of these networks consists of inflow and outflow reactions with reaction rate constants $\k_{i}$ and $\cc_i$, respectively, and a non-flow reaction with reaction rate constant $\ell$. Using the multinomial notation $x^a=x_1^{a_1}\cdots  x_n^{a_n}$, the system of ODEs \eqref{eq:ODE} is: 
\begin{align*}
\frac{dx_i}{dt} &= (b_i-a_i)\ell x^a- \cc_{i}x_i+\k_i,\qquad i=1,\dots,n.
\end{align*}
Therefore the polynomial system we are interested in studying is
\begin{equation}\label{eq:FewNomSys_Nvar}
 (b_i-a_i)\ell x^a- \cc_{i}x_i+\k_i=0,\qquad i=1,\dots,n.
\end{equation}
Note that $a_i,\; b_i$ are fixed, while the parameters $\ell,\; \cc_i$ and $ \k_i$ can vary. In \cite{Joshi} it was shown that this system has at least two positive solutions for some positive $\ell,\cc_i$ and $\k_i$ if and only if 
$$ \sum_{b_i>a_i} a_i>1. $$
We will now  use Gale duality to determine the precise number of positive solutions as well as to understand what parameter regions contain at least two solutions. Let $\sign(i)=\sign(b_i-a_i)$ denote the sign function applied to $b_i-a_i$, that is
\begin{equation}
\sign(i):= \begin{cases} 1 & \textrm{if } b_i>a_i \\ -1  & \textrm{if } a_i>b_i \\ 0 & \textrm{if } a_i=b_i. \end{cases}
\label{eq:signFunction}
\end{equation}We note that we can eliminate the dependence of \eqref{eq:FewNomSys_Nvar} on $b_1,\dots,b_n$ and $\ell$ (up to a term for the sign of $(b_i-a_i)$) by dividing the $i^{th}$ equation by the absolute value of $(b_i-a_i)\ell$ whenever $b_i\neq a_i$. 
With this in mind we may rewrite \eqref{eq:FewNomSys_Nvar} as:
\begin{equation}\label{eq:FewNomSys_Nvar_simp}
\sign(i) x^a-c_{i}x_i+k_{i}=0, \qquad i=1,\dots,n,
\end{equation}
where $c_i:=\cc_i$ and $k_i:=\k_i$ if $b_i=a_i$, and 
$$ c_i:= \frac{\cc_i}{|b_i-a_i|\ell}>0\;\;{\rm and}\;\;  k_i:= \frac{\k_i}{|b_i-a_i|\ell}>0\;\;\; {\rm whenever}\;\;b_i\neq a_i.$$
We can write \eqref{eq:FewNomSys_Nvar_simp} in the form $C \cdot x^W=0$ where 
$x^W=[x_1,\dots,x_n,x^a,1]^T$,
and
\begin{equation*}
C=\begin{bmatrix}
 -c_{1}  & \cdots & 0 & {\rm sgn}(1)& k_{1}\\
 \vdots  & \ddots & \vdots & \vdots& \vdots    \\
 0 & \cdots &-c_{n} & {\rm sgn}(n)& k_{n}\\
\end{bmatrix}\hspace{-0.1cm},  \quad 
W = \begin{bmatrix}
1  & \dots & 0 & a_1 &0 \\ 
\vdots   & \ddots & \vdots & \vdots  & \vdots\\
0 &  \dots & 1 & a_n&0
\end{bmatrix}\hspace{-0.1cm}.
\end{equation*} 
Now compute a Gale dual system to \eqref{eq:FewNomSys_Nvar_simp} using Proposition \ref{Prop:GaleDual2} (see also \S\ref{sec:galedual}). In the notation of Proposition \ref{Prop:GaleDual2} applied to the system \eqref{eq:FewNomSys_Nvar_simp}, we have that $l=1$ and hence any Gale dual system depends on one variable $y$. One can easily check that the following choices of Gale dual matrices $ D\in \RR^{(n+2)\times 2}$ and $Q\in \ZZ^{(n+2)\times 2}$ satisfy the requirements of Proposition \ref{Prop:GaleDual2}:
\begin{equation}
D=\begin{bmatrix}
 \frac{{\rm sgn}(1)}{c_{1}}&\frac{k_{1}}{c_{1}} \\
\vdots & \vdots \\
\frac{{\rm sgn}(n)}{c_{n}}& \frac{k_{n}}{c_{n}}\\
 1 &0 \\
0 &1
\end{bmatrix},\;\;\;\;Q=\begin{bmatrix}
a_1  &0\\
\vdots &\vdots  \\
a_n &0 \\
-1 &0\\
0& 1
\end{bmatrix}.
\end{equation}
In the notation of \S\ref{sec:galedual}, we have
$$ d_i(y)= \tfrac{k_i}{c_i}+\tfrac{\sign(i)}{c_i}y,\quad i=1,\dots,n,\qquad d_{n+1}(y)=y.   $$
Hence the Gale dual system to \eqref{eq:FewNomSys_Nvar_simp} in $\RR[y]$ is given by 
\begin{equation}
 \tfrac{1}{y} \prod_{i=1}^n \left(\frac{{\rm sgn}(i)y +k_i}{c_i} \right)^{\hspace{-1mm}a_i}=1\;\;\;\;{\rm where}\;\; ({\rm sgn}(i)y +k_i) >0,\textrm{ for all }i\ \textrm{and}\ y>0. \label{eq:GD_FirstVer}
\end{equation}
We rewrite this as
\begin{equation}
g(y)= \prod_{i=1}^n \left({\rm sgn}(i)y +k_i \right)^{a_i}-\kappa y=0, {\rm \;\; for\;\;} y\in (0,k_-), \;\; {\rm with \;}\kappa=c^a,\label{GaleSysSimpNet}
\end{equation}
where\begin{equation}
k_-:= \begin{cases} +\infty & \textrm{if }\sign(i)\geq 0 \ \textrm{for all }i=1,\dots,n, \\
\min \{k_i \mid \sign(i)=-1\} & \textrm{otherwise.}
\end{cases}\label{eq:kminusDef}
\end{equation}
Similarily we define \begin{equation}
k_+:= \begin{cases} -\infty & \textrm{if }\sign(i)\leq 0 \ \textrm{for all }i=1,\dots,n, \\
\min \{k_i \mid \sign(i)=1\} & \textrm{otherwise.}
\end{cases}\label{eq:kplusDef}
\end{equation}

\begin{remark}
The indices $i$ in \eqref{eq:GD_FirstVer} for which $\sign(i)=0$ (that is, $b_i=a_i$) or for which $a_i=0$ play no role in the number of solutions of the equation \eqref{eq:GD_FirstVer}. {It is clear from \eqref{eq:FewNomSys_Nvar} that if $\sign(i)=0$, then $x_i= \k_i/\cc_i$ at steady state. Replacing $\ell$ with $\ell ( \k_i/\cc_i)^{a_i}$ and removing $X_i$ yields a new reaction network of the form \eqref{eq:myreaction} with $n-1$ species . In this way we can remove all species with $\sign(i)=0$ and can assume, without loss of generality, that $\sign(i)\neq 0$ for all $i$.  }
\label{remark:a_i_pos}
\end{remark}
By Proposition \ref{Prop:GaleDual2} there is a bijection between positive steady states of the network \eqref{eq:myreaction} and the roots of $g$ in the interval $(0,k_-)$. 
 Given this we study the roots of $g$ for varying values of the parameters $\kappa$, and $k_i$ for $\sign(i)\neq 0$, and deduce the possible number of positive steady states of \eqref{eq:myreaction}. Subsequently, we consider the original steady state equations and the stability of the steady states. In particular, we determine which networks have the capacity for \textit{bistability}, that is the existence of three positive steady states, two asymptotically stable and one unstable.

\subsection{Number of steady states}\label{subsec:numberSteady_1}
In light of Remark \ref{remark:a_i_pos} we may assume, without loss of generality, that $\sign(i)\neq 0$. 
By defining
 \begin{equation}
h(y)=\prod_{i=1}^n \left({\rm sgn}(i)y +k_i \right)^{a_i}, 
\end{equation} 
we may rewrite \eqref{GaleSysSimpNet} as 
\begin{equation} \label{eq:hKappaDef2}
g(y)=h(y)-\kappa y=0,\qquad{\rm where}\;\; y\in (0,k_-).
\end{equation} 
From \eqref{eq:hKappaDef2} we see that the solutions to the Gale dual system correspond to the intersection points of the polynomial $h(y)$ with the line $\kappa y$ in the interval $(0,k_-)$. Observe that the polynomial $h(y)$ has only real roots occurring at $-\sign(i)k_i$ of multiplicity $a_i$ and recall that $\kappa>0$. Let 
\begin{equation}\label{eq:signs0}
a_+:=\sum_{\sgn(i)=1} a_i,  \;\; {\rm and}\;\;  a_-:=\sum_{\sgn(i)=-1} a_i.
\end{equation}
{Our analysis of the number of roots of $g$ can be broken into three cases:
\begin{enumerate}[(i)]
\item $h$ has no negative real roots ($a_+=0$), 
\item $h$ has negative real roots but no positive real roots  ($a_+>0$ and $a_-=0$), 
\item $h$ has both positive and negative real roots ($a_+>0$ and $a_->0$). 
\end{enumerate}
}
In the first case we will see that there must be exactly one positive steady state, in the second case that there are at most two positive steady states, and in the third we will see that there are at most three positive steady states. 

\begin{remark}
In \cite{bihan2016descartes} the authors study the positive real solutions of systems of $n$ polynomial equations in $n$ variables which have $n+1+1$ monomials (i.e.~when $l=1$ in the notation of \S\ref{sec:galedual}). By \cite[Theorem~3.3]{bihan2016descartes} there are at most four positive real solutions to the system of steady state equations \eqref{eq:FewNomSys_Nvar} of the family of networks considered in this section. 
\end{remark}

We start with a simple lemma that applies to the polynomial $h(y)$; the proof is included for completeness. 

\begin{lemma}\label{lem:Thelemma}
Let $p(y)=\lambda \prod_{i=1}^n (y-\alpha_i)^{a_i}$ be an arbitrary polynomial with only real roots occurring at $\alpha_1<\dots<\alpha_n$. 
Then the derivative $p'(y)$ has only real roots; specifically there is one root of multiplicity one in each interval $(\alpha_i,\alpha_{i+1})$ for $i=1,\dots,n-1$, and the remaining roots occur at $\alpha_i$ with multiplicity $a_i-1$ for $i=1,\dots,n$.
\end{lemma}
\begin{proof}

We know that  $p'(y)$ is a polynomial of degree $(a_1+\cdots +a_n)-1$
and that for $j=1,\dots,n$, $\alpha_j$ is a root of  $p'(y)$ of multiplicity $a_j-1$. This gives $(a_1-1)+\cdots + (a_n-1)$ roots of $p'(y)$. 
Now by Rolle's theorem, $p'(y)$ must have a root in each of the intervals $(\alpha_r,\alpha_{r+1})$, for $r=1,\dots,n-1$, since $p(\alpha_j)=0$ for all $j=1,\dots,n$.
Since there are $n-1$ such intervals, we have found all  $(a_1+\cdots +a_n)-1$ roots of $p'(y)$. 
\end{proof}

\subsubsection{No negative real roots} The first case listed above is straightforward to analyze.

\begin{theorem}
{If $a_+=0$, }then the Gale dual system \eqref{eq:hKappaDef2}, $g(y)=0$ for $y\in (0,k_-)$, has exactly one solution of multiplicity one. Equivalently, the network \eqref{eq:myreaction}  has exactly one positive non-degenerate steady state for all choices of parameters {$\ell$, $\cc_i$, and $\k_i$ for all $i$.}
\end{theorem}
\begin{proof}
If $a_-=0$, then $h$ is a constant function with $h(0)>0$. If $a_->0$, then $h$ has only positive real roots and $k_-$ is the smallest root. 
By Lemma \ref{lem:Thelemma} applied to $h$, $h'$ has no root smaller than $k_-$. Since $h(0)>0$ and $h(k_-)=0$, we conclude that $h$ is strictly decreasing and positive in the interval $(0,k_-)$. In both cases, $h$ must meet any line through the origin with positive slope at exactly one point (with multiplicity one) in the interval $(0,k_-)$. 
\end{proof}

\subsubsection{No positive real roots and at least one negative real root}
Consider the case $a_+> 0$ and $a_-=0$, it follows that $b_i>a_i\geq 0$ and hence, that $\sign(i)=1$ for all $i$ and $k_-=+\infty$.
From \eqref{GaleSysSimpNet}  we see immediately that $g(y)$ has no positive roots if 
$$
k^a\cdot \sum_{i=1}^n  \frac{a_i}{k_i}  \geq  \kappa,
$$ 
since in this case all non-zero coefficients of $g$ are positive. If this inequality does not hold, then by the Descartes' rule of signs, we conclude that $g(y)$ has at most two positive roots. 

\begin{theorem}
{ Suppose $a_-=0$ and $a_+>0$.}
 Then for any choice of the parameters $\kappa, \; k_i$  the Gale dual system \eqref{eq:hKappaDef2}, $g(y)=0$ for $y>0$, has 
\begin{itemize}
\item either no solution or one solution of multiplicity one, if $a_+=1$;
\item either no solution or two solutions (counted with multiplicity), if $a_+>1$.
\end{itemize}
In particular, the reaction network \eqref{eq:myreaction} admits at most two positive non-degenerate steady states  if $a_+>1$ and at most one otherwise.
\end{theorem}
\begin{proof}
If $a_+=1$, then $h$ is a line that intersects the line $\kappa y$ in at most one point with multiplicity one. Since $h(0)>0$, by choosing $\kappa$ larger than the slope of $h$ the two lines intersect, and by choosing it smaller, the two lines do not intersect in $(0,+\infty)$. 

Assume now $a_+ >1$. Then $\lim_{y\to +\infty}g(y)=+\infty$ and since $g(0)>0$, then $g$ has either zero or two positive roots (with multiplicity).
The polynomial $h$ has no positive roots, and $h'$ also has no positive roots by Lemma \ref{lem:Thelemma}. Since $h(y)$ tends to $+\infty$, $h(y)$ is increasing and positive for all $y>0$. Further, by applying Lemma \ref{lem:Thelemma} to $h'$, we conclude that $h$ has no positive inflection points and $h''(y)>0$ for $y>0$. 

For any choice of $k_i$, there exists a value $\kappa'$ such that $h$ and $\kappa' y$ intersect in exactly one point, tangentially.
To see this, note that the slope of the tangent line to $h$ at a point $y$ increases towards infinity as $y$ increases. Thus for $y$ large enough, the tangent line intersects the vertical axis at a negative value. Since the tangent line at $y=0$ intersects at a positive value, namely $h(0)$, then by the continuity of $h'$, there exists $y_0>0$ such that the intercept of the tangent line is $0$. The slope of this line is $\kappa'$. Since $h''(y_0)>0$, any $\kappa>\kappa'$ gives a choice of parameters for which the two curves intersect transversely at two points. 
\end{proof}

The proof of the previous theorem is constructive. Choose $k_i$ arbitrarily and solve the equation
$ h(y_0) = h'(y_0)y_0$
for $y_0>0$. If we then choose $\kappa$ larger than  $h'(y_0)$, this gives two positive
\begin{wrapfigure}{r}{0.49\textwidth}\vspace{-0.05in}

  \begin{center}
\begin{tikzpicture}[very thick,scale=0.87, every node/.style={scale=0.89}]
\begin{axis}[
    axis lines = middle,
        ymax=450000,
        xmax=0.85,
        xmin=-0.05,
        ymin= -50000,
    ticks=none,
]
\addplot [
    domain=-0.05:0.6, 
    samples=150, 
    color=oran,smooth, very thick
]
{(x+2)^3*(x+1)^4*(x+8)^2*(x+6)^2} node[above,scale=1.2] {$h(y)$};
\addplot [
    domain=0:0.71,
    samples=150, 
    color=MyCiteColor,smooth, very thick]
    {450000*x} node[above,scale=1.2] {${\kappa y}$};
 \addplot [
    domain=0:0.75,
    samples=150, 
    color=MyRed,smooth, very thick]
    {284907.9238*x} node[above,scale=1.2] {$h'(y_0)y$};   
     \addplot[
    color=MyRed,
    only marks,
    mark=diamond*,
    mark size=1.5pt,
    ]
    coordinates {
    (0.1885498,0)
    } node[below,scale=1.2] {$y_0$};
    \addplot[
    color=Red2,
    only marks,
    mark=*,
    mark size=1.1pt,
    ]
    coordinates {
    (0.05774914590,25987.11566)
    };
     \addplot[
    color=Red2,
    only marks,
    mark=*,
    mark size=1.1pt,
    ]
    coordinates {
    (0.4703666616,211664.9977)
    };
\end{axis}
\end{tikzpicture}  \end{center}\vspace{-0.13in}
  \caption{\footnotesize To obtain two positive intersections we solve $
 h(y_0) - h'(y_0)y_0
$ for $y_0>0$ and take $\kappa>h'(y_0)$.\label{fig:ExAllPos}}\vspace{-0.21in}
\end{wrapfigure}
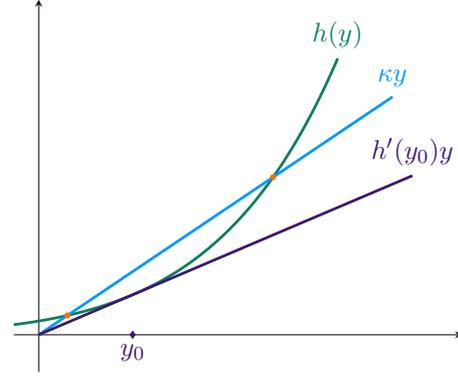\noindent
non-degenerate steady states. 
For example, let $a=[3,4,2,2]$ with $b_i>a_i$. Choose $k=[2,1,8,6]$, then $
h(y)=(y+2)^3(y+1)^4(y+8)^2(y+6)^2
$. To find a value of $\kappa$ such that the Gale dual system $g(y)=h(y)-\kappa y$ has two positive solutions, we first find the positive root of the polynomial $
 h(y_0) - h'(y_0)y_0
$. We obtain (approximately) $y_0\simeq 0.189$ and $h'(y_0)\simeq 284908$. Hence, setting $\kappa>284908$ will yield two positive solutions. If we choose $\kappa=450000$ the two positive solutions of $g(y)=0$ are (approximately) {$y\simeq 0.058$} and {$y\simeq 0.47$}, see Figure \ref{fig:ExAllPos}. 

\subsubsection{At least one negative and one positive real root}Assume that $a_+>0$ and $a_->0$, so that $h$ has both positive and negative roots. Then $k_-$ is the first positive root of $h$.
Since $g(k_-)<0$ and $g(0)>0$, we must have at least one solution and must have an odd number of solutions ({counted with multiplicity}). 
We show that in fact we have at most three solutions in this case, and that this number can be achieved for appropriate choices of $k_i$ and $\kappa$. 

\begin{theorem}
{Suppose that $a_+>0$ and $a_->0$.}
Consider the Gale dual system $g(y)=0$ for $y\in (0,k_-)$ as in \eqref{eq:hKappaDef2}. We have that:
\begin{enumerate}[(i)]
\item There are at most three solutions if $a_+>1$, and exactly one otherwise.
\item There are three solutions if and only if $h'$ has a root $\xi$ in $(0,k_-)$, $h''$ has a root $y^*$ in $(0,\xi)$, and $t_h(0)\leq 0$ where $t_h(y)$ is the tangent line to $h$ at $y^*$.
\item The polynomial $h'$ has a root $\xi$ in $(0,k_-)$ and $h''$ has a root $y^*$ in $(0,\xi)$ if and only if
$$\gamma=\sum_{j=1}^n \frac{{\rm sgn}(j)a_j}{k_{j} }>0 \quad {\rm and} \quad \theta=\left( \gamma^2-\sum_{j=1}^n \frac{a_j}{(k_{j})^2}\right)>0.$$
\item {When $a_+>1$,} there exist $k_i$ and $\kappa$ such that there are exactly three solutions of multiplicity one.
\end{enumerate}\label{thm:atMost3Sol} 
In particular,  the reaction network \eqref{eq:myreaction} admits at most three positive non-degenerate steady states if $a_+>1$ and at most one otherwise.
\end{theorem}
\begin{proof}

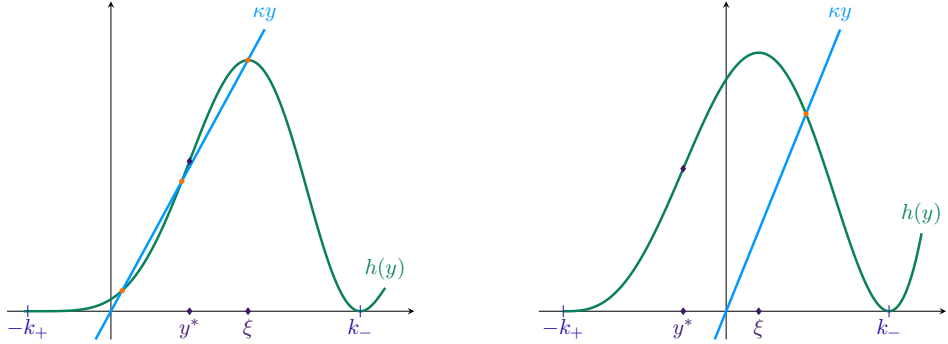
\begin{figure}[t!]\hspace{-5mm}
\begin{subfigure}[h]{0.32\linewidth}
\begin{tikzpicture}[thick,scale=0.79, every node/.style={scale=0.79}]
\begin{axis}[
    axis lines = middle,
        ymax=61000000,
        xmax=7.3,
        xmin=-2.5,
    ticks=none,
]
\addplot [
    domain=-2.1:6.6, 
    samples=150, 
    color=oran,smooth, very thick
]
{(x+4)^3*(x+2)^4*(-x+8)^2*(-x+6)^2} node[above, scale=1.2] {$h(y)$};
\addplot [
    domain=-0.37:3.7,
    samples=150, 
    color=MyCiteColor,smooth, very thick]
    {1.5*10^7*x} node[above, scale=1.2] {$\kappa y$};
    \addplot[
    color=MyBlue,
    only marks,
    mark=|,
    mark size=2.9pt,
    ]
    coordinates {
    (-2,0)
    } node[below, scale=1.2] {$-k_+$};
    \addplot[
    color=MyBlue,
    only marks,
    mark=|,
    mark size=2.9pt,
    ]
    coordinates {
    (6,0)
    } node[below, scale=1.2] {$k_-$};
        \addplot[
    color=MyRed,
    only marks,
    mark=diamond*,
    mark size=1.5pt,
    ]
    coordinates {
    (1.891359858,0)
    } node[below, scale=1.2] {$y^*$};
      \addplot[
    color=MyRed,
    only marks,
    mark=diamond*,
    mark size=1.5pt,
    ]
    coordinates {
    (3.298940941,0)
    } node[below, scale=1.2] {$\xi$};
     \addplot[
    color=MyRed,
    only marks,
    mark=diamond*,
    mark size=1.5pt,
    ]
    coordinates {
    (1.891359858,2.953503743*10^7)
    }; \addplot[
    color=Red2,
    only marks,
    mark=*,
    mark size=1pt,
    ]
    coordinates {
    (1.704785411,2.557178116*10^7)
    };
    \addplot[
    color=Red2,
    only marks,
    mark=*,
    mark size=1pt,
    ]
    coordinates {
    (.2707503407,4.061255110*10^6)
    };
    \addplot[
    color=Red2,
    only marks,
    mark=*,
    mark size=1pt,
    ]
    coordinates {
    (3.295358068,4.943037102*10^7)
    };
\end{axis}
\end{tikzpicture}
\end{subfigure}\hspace{19mm}
\begin{subfigure}[h]{0.32\linewidth}
\begin{tikzpicture}[thick,scale=0.79, every node/.style={scale=0.79}]
\begin{axis}[
    axis lines = middle,
            ymax=770,
        xmax=1.35,
        xmin=-1.15,
    ticks=none,
]
Below the red parabola is defined
\addplot [
    domain=-1:1.2, 
    samples=100, 
    color=oran,smooth, very thick
]
{(x+4)^3*(x+1)^3*(-x+1)^2*(-x+3)^2}  node[above, scale=1.2] {$h(y)$} ;
\addplot [
    domain=-0.07:0.7, 
    samples=100, 
    color=MyCiteColor,smooth, very thick
    ]
    {1000*x} node[above, scale=1.2] {$\kappa y$};
        \addplot[
    color=MyBlue,
    only marks,
    mark=|,
    mark size=2.9pt,
    ]
    coordinates {
    (-1,0)
    }node[below, scale=1.2] {$-k_+$};
            \addplot[
    color=MyBlue,
    only marks,
    mark=|,
    mark size=2.9pt,
    ]
    coordinates {(1,0)
    }node[below, scale=1.2] {$k_-$};
    
     \addplot[
    color=MyRed,
    only marks,
    mark=diamond*,
    mark size=1.5pt,
    ]
    coordinates {
    (-.2635471014,354.3040402)
    };
         \addplot[
    color=MyRed,
    only marks,
    mark=diamond*,
    mark size=1.5pt,
    ]
    coordinates {
    (-.2635471014,0)
    }node[below, scale=1.2] {$y^*$};
             \addplot[
    color=MyRed,
    only marks,
    mark=diamond*,
    mark size=1.5pt,
    ]
    coordinates {
    (0.2000000000,0)
    }node[below, scale=1.2] {$\xi$};
        \addplot[
    color=Red2,
    only marks,
    mark=*,
    mark size=1pt,
    ]
    coordinates {
    (.4903328354,490.3328354)
    };
\end{axis}
\end{tikzpicture}
\end{subfigure}\vspace{-0.15in}
\caption{\small Plots of the intersection of $h$ and the line $\kappa y$ in the interval $(-k_+,k_-)$ for $a_+>1$ and $a_->0$. Let $\xi$ be the local maximum of $h$ in the interval $(0,k_-)$. For $h$ and $\kappa y$ to meet at more than one point, $h$ must have an inflection point $y^*\in (0,\xi)$. \label{Fig:3plots}}\vspace{-0.17in}
\end{figure}
By Lemma~\ref{lem:Thelemma}, $h'$ has exactly one root $\xi$ of multiplicity one in the interval $(-k_{+},k_-)$. So $h''(\xi)\neq 0$ and $\xi$ is a local extremum. Since $h(0)>0$  and $h(k_-)=0$, it follows that $\xi$ is a local maximum, $h''(\xi)<0$, and $h$ decreases in the interval $(\xi,k_-)$. This gives that $h$ and $\kappa y$ intersect in at least one point in the interval $(0,k_-)$.

\color{black}
First we prove (i). 
Consider the case $a_+=1$. {Then $h$ has exactly one negative real root of multiplicity $1$} and $\xi$ is the smallest root of $h'$.  
By Lemma~\ref{lem:Thelemma} applied to $h'$, $h''$ has no root smaller than $\xi$, and hence $h$ has no inflection point smaller than $\xi$. Therefore, 
the function $h$ in the interval $(0,k_-)$ is positive, and either strictly decreasing, or increasing with $h''<0$ up to $\xi>0$ and afterwards decreasing. In both cases we say that $h$ has property $\dagger$; the intersection of such a function with a line through the origin $\kappa y$ is transversal and consists of exactly one point. 
Indeed, in the first case, a strictly decreasing function and a strictly increasing function intersect. In the second case, if 
 $h$ and $\kappa y$ intersect in $(0,\xi)$ then, since $h$ is concave in this region, there can be only one intersection; further there can be no intersections in $[\xi ,k_-)$ since $h$ is decreasing in this interval. Similarly if $h$ and $\kappa y$ do not intersect in $(0,\xi)$, then they must intersect in $[\xi ,k_-)$, but can intersect no more than once since $h$ is decreasing and $\kappa y$ is increasing in $(\xi, k_-)$. 

\color{black}
Now assume $a_+>1$, then by Lemma~\ref{lem:Thelemma}, $h'$ has at least one more root $\xi'$ smaller than $\xi$. Therefore, $h''$ has one root $y^*$ in the interval $(\xi',\xi)$. If $y^*\leq 0$, then $h$ again has property $\dagger$ and intersects the line through the origin $\kappa y$ in exactly one point in the interval $(0,k_-)$. 
If $y^*>0$, then a line through the origin may intersect $h$ in at most three points, and at least one. See Figure \ref{Fig:3plots}.

We now prove (ii). Assume that $\xi, y^*>0$ and consider the tangent line to $h$ at $y^*$, given by $t_h(y)=h(y^*) + h'(y^*)(y-y^*).$
Since $h'(y)$ increases for $y<y^*$ and $h'(y)$ decreases for  $y^*<y<{\xi}$, then 
 
\begin{equation}\label{eq:inequalities}
h(y)-t_h(y)>0, \quad \textrm{ for }{-k_+}<y<y^*,\quad \textrm{and}\quad h(y)-t_h(y)<0, \quad\textrm{ for }y^*<y<k_-,\end{equation}
and \color{black}
 the only intersection of the line $t_h$ with $h$ in the interval $(0,k_-)$ occurs at $y^*$. Let $L$ be any line through $(y^*,h(y^*))$ with positive slope. If the slope of $L$ is larger than $h'(y^*)$, then it intersects $h$ only at $y^*$. On the other hand if $L(y)$ has positive slope smaller than $h'(y^*)$ and if $L(0)<h(0)$, then $L$ intersects $h$ in three points. 
Therefore, if $t_h(0)\leq 0$, then by decreasing the slope of the line $L=t_h$ until $L(0)=0$, we obtain a choice of $\kappa$ such that $\kappa y$ intersects $h$ at three points. This shows the reverse implication of (ii). 

To prove the forward implication, assume that $h$ and $\kappa y$ intersect at three points in the interval $(0,k_-)$. {
Then $h$ does not have property $\dagger$, which gives that $\xi \in  (0,k_-)$ and $h''$ has a root $y^*$ in  $(0,\xi)$. 
Let $t_h(y)$ be the tangent line to $h$ at $y^*$. We show that if $t_h(0)>0$, then $h$ and $\kappa y$ intersect in one point in $(0,k_-)$, yielding a contradiction.
If the line $\kappa y$ meets $h$ in the interval $(0,y^*]$, then it also meets $t_h(y)$ in that interval, and hence $\kappa$ is greater than the slope of $t_h(y)$. Using \eqref{eq:inequalities} it follows that $\kappa y$ and $h(y)$ intersect once in $(0,k_-)$. If $h$ and $\kappa y$ only intersect in the interval $(y^*,k_-)$, then intersection can only occur in one point, since  $h$ has property $\dagger$ in $(y^*,\xi)$.} This concludes the proof of (ii).

Now we prove (iii). In the interval $(-k_+,k_-)$ we may write 
\begin{align*}
h'(y) & = h(y) \sum_{j=1}^n \frac{{\rm sgn}(j)a_j}{{\rm sgn}(j)y +k_{j} }, \\
h''(y)&=h'(y) \sum_{j=1}^n \frac{{\rm sgn}(j)a_j}{{\rm sgn}(j)y +k_{j} }-h(y) \sum_{j=1}^n \frac{a_j}{({\rm sgn}(j)y +k_{j} )^2} \numberthis\label{eq:h''ofy}\\
&=h(y) \left( \left(\sum_{j=1}^n \frac{{\rm sgn}(j)a_j}{{\rm sgn}(j)y +k_{j} }\right)^{\hspace{-0.15cm} 2}-\sum_{j=1}^n \frac{a_j}{({\rm sgn}(j)y +k_{j} )^2}\right).
\end{align*}
Since $h'>0$ in $(-k_+,\xi)$ and $h'<0$ in $(\xi,k_-)$,  $\xi$ belongs to  $(0,k_-)$ if and only if $h'(0)>0$. Using that  $h'(0)=h(0)\cdot \gamma$ and that  $h(0)>0$, we obtain that $\xi\in (0,k_-)$ if and only if $\gamma>0$. 
Similarly, assume now that $\xi \in (0,k_-)$. 
Recall that by  Lemma \ref{lem:Thelemma}, if $h''$ has a root $y^*\in (-k_+,\xi)$, then it has multiplicity one.
Since $h''(\xi)<0$, it follows that  $h''$ has a root in the interval $(0,\xi)$ if and only if $h''(0)>0$. By \eqref{eq:h''ofy} we have $h''(0)=h(0)\cdot \theta$, which gives 
that $h''$ has a root $y^*\in (0,\xi)$ if and only if $\theta>0$. 
This concludes the proof of (iii).

\color{black}
We now prove (iv). We can reduce to the case $n=2$ by setting $k_i=k_+$ for all $i$ such that $\sign(i)=1$, and setting $k_i=k_-$ for all $i$ such that $\sign(i)=-1$. With this substitution we may now study the system:
$$
g(y)=(y+k_+)^{a_+}(-y+k_-)^{a_-}-\kappa y=0 , \;\;\; 0<y<k_-.$$
We wish to show that we may choose $k_-,\;k_+,$ and $\kappa$ such that $g$ changes sign three times in the interval $(0,k_-)$. To do this, we will consider the values $g(k_+)$ and $g({2k_+})$. 
We have that 
$$
g(k_+)=2^{a_+}k_+^{a_+}(k_--k_+)^{a_-}-\kappa k_+, \;\;{\rm and} \;\; g(2k_+)=3^{a_+}k_+^{a_+}(k_--2k_+)^{a_-}-2\kappa k_+.
$$
Since $g(0)>0$ and $g(k_-)<0$ it suffices to show that there exists a choice of $k_-,\;k_+,$ and $\kappa$ such that $2k_+<k_-$, $g(k_+)<0$ and $g(2k_+)>0$, or equivalently, that $2k_+<k_-$ and 
\begin{equation}
2^{a_+}k_+^{a_+-1}(k_--k_+)^{a_-}<\kappa <\tfrac{3^{a_+}}{2}k_+^{a_+-1}(k_--2k_+)^{a_-}.\label{eq:EqExists3Sols}
\end{equation} 
The inequality $2^{a_+}k_+^{a_+-1}(k_--k_+)^{a_-}<\frac{3^{a_+}}{2}k_+^{a_+-1}(k_--2k_+)^{a_-}$ is equivalent to 
\begin{equation}
\frac{k_--k_+}{k_--2k_+}< \alpha,\;\; \textrm{with }\;\;\ \alpha= \left( \tfrac{3^{a_+}}{2^{1+a_+}} \right)^{\frac{1}{a_-}}.\label{eq:simpinEqExists3Sols}
\end{equation}
Since  $a_+>1$ we have that $\alpha>1$. For a fixed $k_+$, it is clear that we can make $\frac{k_--k_+}{k_--2k_+}$ arbitrarily close to one  for $k_-$ large enough, and in particular less that $\alpha$. It follows that there exist $k_-,$ and $k_+,$ such that $2k_+<k_-$, and the inequality \eqref{eq:simpinEqExists3Sols} holds; hence we can choose $\kappa$ such that the inequality \eqref{eq:EqExists3Sols} holds.
\end{proof}

\begin{example}\label{ex:2}
\begin{figure}[b!]\hspace{-1.4in}
\begin{subfigure}[h]{0.15\linewidth}
\begin{tikzpicture}[thick,scale=0.85, every node/.style={scale=0.85}]
\begin{axis}[
    axis lines = middle,
    xmin=-2.45,
    xmax=5,
    ymax=5000000,
    ticks=none,
]
\addplot [
    domain=-2.1:4.4, 
    samples=150, smooth, very thick,
    color=oran,
]
{(x+4)^3*(x+2)^4*(-x+4)^2*(-x+6)^2} node[above,scale=1.2] {$h(y)$};
\addplot [
    domain=-0.3:1.9, 
    samples=150, 
    color=MyCiteColor,smooth, very thick
    ]
    {2.304937796*10^6*x}  node[above,scale=1.2] {$\kappa y$};
    \addplot[
    color=MyBlue,
    only marks,
    mark=|,
    mark size=2.9pt,
    ]
    coordinates {
    (-2,0)
    } node[below,scale=1.2] {$-k_+$};
        \addplot[
    color=MyBlue,
    only marks,
    mark=|,
    mark size=2.9pt,
    ]
    coordinates {
(4,0)
    } node[below,scale=1.2] {$k_-=4$};
     \addplot[
    color=MyRed,
    only marks,
    mark=diamond*,
    mark size=1.5pt,
    ]
    coordinates {
    (.9170770853,2.113805636*10^6)
    };
         \addplot[
    color=MyRed,
    only marks,
    mark=diamond*,
    mark size=1.5pt,
    ]
    coordinates {
    (.9170770853,0)
    } node[below,scale=1.2] {$y^*$};
         \addplot[
    color=Red2,
    only marks,
    mark=o,
    mark size=2.5pt,
    ]
    coordinates {
    (.9170770853,2.113805636*10^6)
    };
\end{axis}
\end{tikzpicture}
\end{subfigure}\hspace{1.85in}
\begin{subfigure}[h]{0.15\linewidth}
\begin{tikzpicture}[thick,scale=0.85, every node/.style={scale=0.85}]
\begin{axis}[
    axis lines = middle,
        xmin=-2.55,
    xmax=7.7,
    ymax=40000000,
    ticks=none,
]
\addplot [
    domain=-2.1:6.9, 
    samples=150, 
    color=oran,smooth,very thick
]
{(x+4)^3*(x+2)^4*(-x+6)^2*(-x+7)^2} node[above,scale=1.2] {$h(y)$};
\addplot [
    domain=-0.52:3.3, 
    samples=150, 
    color=MyCiteColor,smooth,very thick
    ]
    {1.065687498*10^7*x} node[above,scale=1.2] {$\kappa y$};
        \addplot[
    color=MyBlue,
    only marks,
    mark=|,
    mark size=2.9pt,
    ]
    coordinates {
    (-2,0)
    } node[below,scale=1.2] {$-k_+$};
    \addplot[
    color=MyBlue,
    only marks,
    mark=|,
    mark size=2.9pt,
    ]
    coordinates {
    (6,0)
    } node[below,scale=1.2] {$k_-=6$};
     \addplot[
    color=MyRed,
    only marks,
    mark=diamond*,
    mark size=1.5pt,
    ]
    coordinates {
    (1.728644230,1.842194545*10^7)
    };     \addplot[
    color=MyRed,
    only marks,
    mark=diamond*,
    mark size=1.5pt,
    ]
    coordinates {
    (1.728644230,0)
    } node[below,scale=1.2] {$y^*$};
      \addplot[
    color=Red2,
    only marks,
    mark=o,
    mark size=2.5pt,
    ]
    coordinates {
    (1.728644230,1.842194545*10^7)
    }; 
          \addplot[
    color=Red2,
    only marks,
    mark=*,
    mark size=1pt,
    ]
    coordinates {
    (.3099442711,3.303037348*10^6)
    }; 
              \addplot[
    color=Red2,
    only marks,
    mark=*,
    mark size=1pt,
    ]
    coordinates {
    (2.853148131,3.040564293*10^7)
    }; 
\end{axis}
\end{tikzpicture}
\end{subfigure}\vspace{-0.1in}
\caption{Plots of the intersection of $h$ and the $\kappa y$ in the interval $(-k_+,k_-)$ for different choices of $k_i$ and $\kappa$. See Example \ref{ex:2}. \label{Fig:2plots}\vspace{-0.09in}}
\end{figure}
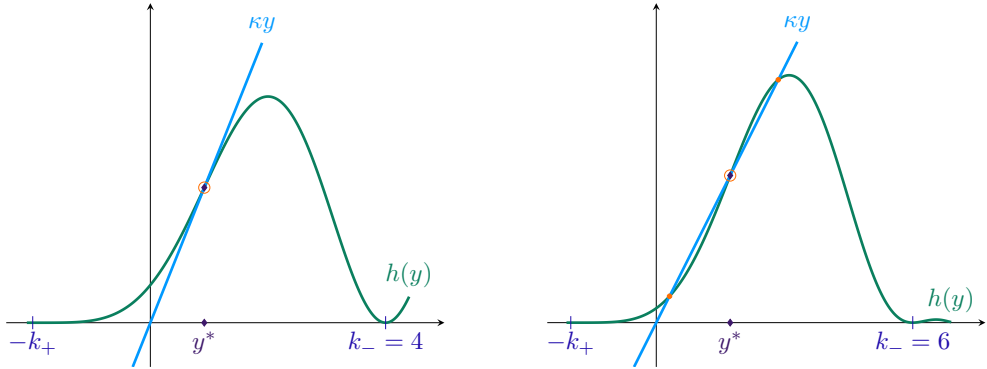
Let $a=[3,4,2,2]$ and suppose that $b_1>3$, $b_2>4$, $b_3=b_4=1$. Then $h(y)=(y +k_{1})^{3}(y +k_{2})^{4}(-y +k_{3})^{2}(-y +k_{4})^{2}$.
By Theorem \ref{thm:atMost3Sol} to obtain three solutions to the Gale dual system we must have that:$$
\gamma=\tfrac{3}{k_{1} }+\tfrac{4}{k_{2} }-\tfrac{2}{k_{3} }-\tfrac{2}{k_{4} }>0\;\;\; {\rm and}\;\;\;\theta= \gamma^2-\left(\tfrac{3}{k_{1}^2 }+\tfrac{4}{k_{2}^2 }+\tfrac{2}{k_{3}^2 }+\tfrac{2}{k_{4}^2 }\right)>0.
$$
Choosing $k=[4,2,4,6]$ gives $k_-=4$, $\gamma=\frac{23}{12}>0$ and $\theta=\frac{145}{144}>0$; in this case the one root of $h'$ in $(0,4)$ is $\xi=2$ and the inflection point in the interval $(0,2)$ is $y^*\simeq 0.917$. The tangent line to $h$ at $y^*$ is (approximately):$$t_h(y)=1.985169547\cdot 10^6y+293252.134,\;\;\;{\rm with\;} \ t_h(0)>0.$$
Hence any choice of $\kappa$ gives exactly one intersection of $\kappa y$ and $h$ in the interval $(0,4)$. This is illustrated in the left hand plot in Figure \ref{Fig:2plots}.
On the other hand choosing $k=[4,2,6,7]$ gives $k_-=6$, $\gamma=\frac{179}{84}>0$ and $\theta=\frac{7699}{7056}>0$; in this case the one root of $h'(y)$ in $(0,6)$ is $\xi\simeq 3.1075$ and the inflection point in the interval $(0,\xi)$ is $y^*\simeq 1.73$. The tangent line to $h$ at $y^*$ is (approximately):$$t_h(y)=1.379464137\cdot 10^7y-5.42408176\cdot 10^6,\;\;\;{\rm with\;}t_h(0)<0.$$
From this we see that setting $\kappa=\frac{h(y^*)}{y^*}=1.065687498\cdot 10^7$ gives exactly three intersections of $\kappa y$ and $h$ in the interval $(0,6)$. This is illustrated in the right hand plot in Figure \ref{Fig:2plots}. The three points are (approximately): 
$y\simeq 0.3099, \  1.7286, \  2.8531.$
\end{example}

\noindent A sufficient condition to obtain three positive steady states of \eqref{eq:FewNomSys_Nvar} is given in Proposition \ref{prop:secondSufficentCond}.

\begin{proposition}\label{prop:secondSufficentCond}
{Assume that $n\geq 3$, $a_+>0$ and $a_->0$. 
If there exist $k_{\mu}< k_{\nu}<k_-$ with $\sign(\mu)=\sign(\nu)=1$, $a_\mu\geq 1$, $a_\nu\geq 1$, and }
\begin{align*}
2^{a_\mu}k_{\mu}^{a_\mu-1}\prod_{i=1\atop i\neq \mu}^{n} \left(\sign(i) k_{\mu} +k_i \right)^{a_i}&< \kappa 
<2^{a_\nu}k_{\nu}^{a_\nu-1}\prod_{i=1\atop i\neq \nu}^{n} \left(\sign(i) k_{\nu} +k_i \right)^{a_i} \;\;\;{\rm for\; some\;} \kappa>0,
\end{align*} 
then the Gale dual system \eqref{GaleSysSimpNet} has exactly three solutions.\label{prop:Exact3Sol}
\end{proposition}
\begin{proof}
Let $g$ be as in \eqref{GaleSysSimpNet}, we have that  $g(0)>0$, $g(k_{\mu})<0$,  $g(k_{\nu})>0$ and  $g(k_{-})<0$. Hence, $g$ has at least three roots in the interval $(0,k_-)$; the conclusion follows by Theorem \ref{thm:atMost3Sol}(i).
\end{proof}

\subsection{Stability of the steady states}\label{subsec:StabSteady_1}
We now study the asymptotic stability of the steady states of the reaction network \eqref{eq:myreaction}. 
For any number of species $n$, the embedded network on any subset of species $\{X_{i_1},\dots,X_{i_s}\}$ is of the same type, \eqref{eq:myreaction}, with the number of species equal to $s\leq n$. 
In light of this we first study the case where $n=1$; in this case all signs agree. We then study the case where $n=2$ and both non-zero signs occur. Finally, we extend our conclusions to arbitrary $n$ using Proposition \ref{prop:embedded} and the auxiliary networks described in Lemma \ref{lemma:extending}.

\begin{proposition}\label{prop:stability}
Let $n=1$ and consider the stability of the steady states of \eqref{eq:myreaction}.
\begin{itemize}
\item If $a_1>b_1$, the only positive steady state is asymptotically stable.
\item If $a_1<b_1$ and there is one positive steady state of multiplicity one, then it is asymptotically stable. If there are two positive steady states, then the smallest is asymptotically stable and the largest is unstable.
\end{itemize}
\end{proposition}
\begin{proof}
The ODE system of the network \eqref{eq:myreaction} is 
\begin{align*}
\frac{dx_1}{dt} &=f(x_1),\qquad\textrm{with }\; f(x_1)=  (b_1-a_1)\ell x_1^{a_1}- \cc_{1}x_1+\k_{1}.
\end{align*}
A steady state $x_1^*$ is a solution to the equation $f(x_1^*)=0$. 
It is well known that if $f'(x_1^*)<0$, then the steady state is asymptotically stable, and if $f'(x_1^*)>0$, then it is unstable. 
Since $f(0)>0$, if the network has one positive steady state $x_1^*$, of multiplicity one, then necessarily $f'(x_1^*)<0$ and thus $x_1^*$ is asymptotically stable. If the network has two positive steady states (which must each be of multiplicity one), then the smallest steady state satisfies $f'(x_1^*)<0$ and is asymptotically stable, and the largest satisfies $f'(x_1^*)>0$ and it is unstable. 
\end{proof}

\begin{proposition}\label{prop:stability3}
Let $n=2$ with $0<a_1<b_1$,  $a_2>b_2$ and consider the stability of the steady states of the reaction network \eqref{eq:myreaction}.
\begin{itemize}
\item If there is only one positive steady state of multiplicity one, it is asymptotically stable.
\item If $\cc_{1}<\cc_{2}$ and there are three positive steady states, then two are asymptotically stable and the other is unstable.
\end{itemize}
\end{proposition}
\begin{proof}
The ODE system is of the form
\begin{align*}
\frac{dx_1}{dt} &= \alpha_1 x_1^{a_1}x_2^{a_2}- \cc_{1}x_1+\k_{1}, &
\frac{dx_2}{dt} &= -\alpha_2  x_1^{a_1}x_2^{a_2}- \cc_{2}x_2+\k_{2}.
\end{align*}
with $\alpha_1=(b_1-a_1)\ell>0$ and $\alpha_2=(a_2-b_2)\ell>0$. The associated Jacobian matrix is
$$ J(x) = \begin{bmatrix}
 a_1 \alpha_1  x_1^{a_1-1}x_2^{a_2}- \cc_{1}  &  a_2\alpha_1  x_1^{a_1}x_2^{a_2-1} \\[6pt]
-a_1\alpha_2  x_1^{a_1-1}x_2^{a_2} &  -a_2\alpha_2  x_1^{a_1}x_2^{a_2-1}- \cc_{2}
\end{bmatrix}.$$
Consider a steady state $x=(x_1,x_2)$. Since we have only two variables, if the determinant of $J(x)$ is negative, then the steady state is unstable, and if the determinant is positive and the trace is negative, then the steady state is asymptotically stable \cite{perko}. Compute the trace and determinant of $J(x)$:
\begin{align*}
& \det(J(x)) = -x_1^{a_1-1}x_2^{a_2-1} (a_1\alpha_1\cc_{2} x_2 - a_2\alpha_2 \cc_{1} x_1) + \cc_{1} \cc_{2} ,\\
&\tr(J(x))   = x_1^{a_1-1}x_2^{a_2-1} (a_1\alpha_1 x_2 - a_2\alpha_2 x_1) - \cc_{1}- \cc_{2}.
\end{align*}

Rewriting the two steady state equations, we see that for our steady state $x$ we have that 
$x_2 =\frac {-\alpha_{{2}}\cc_{1}x_{{1}}+\alpha_{{1}}\k_{{2}}+\alpha_{2}\k_{1}}{\alpha_{1}\cc_{2}}$ and hence
\begin{equation}\label{eq:x2}
 f(x_1)=\alpha_1x_1^{a_1} \left( \frac {-\alpha_2\cc_{1}x_1+\k_{{2}}\alpha_{{1}}+\alpha_{{2}}\k_{1}}{\alpha_{{1}}\cc_2} \right) ^{a_{{2}}}-\cc_{1}x_{{1}}+\k_{1}=0.
\end{equation} If the network has one positive steady state $x$ of multiplicity one,  since $f(0)>0$, we must have $f'(x_1)<0$. If the network admits three steady states of multiplicity one, then two of them satisfy $f'(x_1)<0$ and the other $f'(x_1)>0$. 
Hence, it is enough to show that if $f'(x_1)<0$, then the steady state is asymptotically stable, and if $f'(x_1)>0$, then the steady state is unstable. Let $x=(x_1,x_2)$ be a positive steady state.
Using \eqref{eq:x2}, the derivative of $f$ at $x_1$ is
\begin{align*}
f'(x_1) & =a_1 \alpha_1x_1^{a_1-1} x_2^{a_2} -a_2\alpha_2 x_1^{a_1}  x_2^{a_2-1} \tfrac{\cc_{1}}{\cc_{2}}-\cc_{1}  
= \tfrac{-\det(J(x))}{\cc_{2}}.
\end{align*}
Thus $f'(x_1)$ and $\det(J(x))$ have opposite signs at a steady state $(x_1,x_2)$. It follows that if $f'(x_1)>0$, then the determinant is negative and the steady state is unstable. On the other hand, if $f'(x_1)<0$ this implies $\det(J(x))>0$ and we need to show that the trace is negative for the steady state to be asymptotically stable. 
Suppose that $\det(J(x))>0$. We have that \small
$$ \alpha_1a_1 x_1^{a_1-1}x_2^{a_2}<\frac {a_2\alpha_2\cc_{1} x_1^{a_1}{x_2}^{a_2-1}}{\cc_{2}}+\cc_{1}, \;\;{\rm giving,\;\;}\tr(J(x)) <  \frac{x_1^{a_1}{x_2}^{a_2-1}a_2\alpha_2 ( \cc_{1}-\cc_{2}) }{\cc_{2}} - \cc_{2}.
 $$ \normalsize
The last inequality shows that $\tr(J(x))$ is negative if $\cc_{1}<\cc_{2}$. 
\end{proof}

\begin{remark}  Note that the condition $\cc_{1}<\cc_{2}$ does not affect the capacity of the network to have multiple steady states. Specifically, while the parameters $\cc_i$ do appear in $\kappa$, see \eqref{GaleSysSimpNet}, and while the value of $\kappa$ does play a role in determining if multiple steady states may be achieved, requiring that $\cc_{1}<\cc_{2}$ does not constrain the values which $\kappa$ can take. 
\end{remark}
Before stating the main theorem of this section, we prove the following lemma. 
\begin{lemma}\label{lemma:extending}
Let $N$ be a reaction network as in \eqref{eq:myreaction} with $n\geq 2$ and assume $a_1<b_1$, $a_2<b_2$. Let $\tilde{a}=a_1+a_2$ and $\tilde{b}=b_1+b_2$ and consider the reaction network $\tilde{N}$ of type \eqref{eq:myreaction} with $n-1$ species $Y_1,Y_3\dots,Y_n$ and  reactions 
$$ \tilde{a} Y_1 + a_3Y_3 + \dots + a_n Y_n \xrightarrow{\ell}  \tilde{b} Y_1 + b_3Y_3 + \dots + b_n Y_n,\qquad Y_i \ce{<=>[\tilde{\cc}_{i}][\tilde{\k}_{i}]}
0,\quad i=1,3,\dots,n .$$
Let $\beta_i  = \tfrac{b_i-a_i}{\tilde{b}-\tilde{a}}$ and define \begin{equation*}
   \cc_i  =\begin{cases} \beta_i \tilde{\cc}_{1} & \textrm{for } i=1,2 \\ \tilde{\cc}_{i}  & \textrm{for } i=3,\dots,n  \end{cases}  \;\;{\rm and}\;\; \k_i  = \begin{cases} \beta_i  \tilde{\k}_{1} & \textrm{for } i=1,2 \\ \tilde{\k}_{i}  & \textrm{for } i=3,\dots,n  \end{cases}.
\end{equation*}
If $\tilde{y}=(y_1,y_3,\dots,y_n)$ is a steady state of $\tilde{N}$ for the reaction rate constants $\ell,\; \tilde{\cc}_{i},\; \tilde{\k}_{i}$, for $i=1,3,\dots,n$, then $y^*=(y_1,y_1,y_3,\dots,y_n)$ is a steady state of $N$ for the reaction rate constants $\ell$, $\cc_i$ and $\k_i$. 

\end{lemma}
\begin{proof}
For $i\geq 3$, since $\tilde{y}$ is a steady state of $\tilde{N}$ we have
$$0= (b_i-a_i ) \ell y_1^{\tilde{a} } y_3^{a_3}\cdots y_n^{a_n}   - \tilde{\cc}_{i} y_i + \tilde{\k}_{i}=
 (b_i-a_i ) \ell y_1^{a_1} y_1^{a_2}y_3^{a_3}\cdots y_n^{a_n}   - \cc_i y_i + \k_i.  $$
Hence, the equations $\frac{dx_i}{dt}=0$ hold for $N$ at the point $y^*$. Further, we know that 
$0 = (\tilde{b} - \tilde{a} ) \ell  y_1^{\tilde{a} } y_3^{a_3}\cdots y_n^{a_n}  - \tilde{\cc}_{1} y_1 + \tilde{\k}_{1}. $
For $i=1,2$, the steady state equation for $N$ evaluated at $y^*$ is
$$   (b_i-a_i) \ell y_1^{a_1} y_1^{a_2} y_3^{a_3}\cdots y_n^{a_n}   - \cc_i y_i + \k_i=
 \beta_i \big(   (\tilde{b} - \tilde{a} ) \ell  y_1^{\tilde{a} } y_3^{a_3}\cdots y_n^{a_n}  - \tilde{\cc}_{1} y_1 + \tilde{\k}_{1}\big)=0.\quad\quad\quad\quad\qedhere
   $$
\end{proof}

Using Propositions \ref{prop:embedded}, \ref{prop:stability}, \ref{prop:stability3} and Lemma \ref{lemma:extending}, we  obtain the following theorem.

\begin{theorem}\label{thm:stability1}
Consider a reaction network of the type \eqref{eq:myreaction}.
\begin{enumerate}[(i)]
\item There exists a choice of {rate} parameters such that the network has \textcolor{blue}{at least} one positive steady state which is asymptotically stable.
\item If {$a_+>1$ and $a_-=0$}, then there exist {rate} parameter values such that there is one asymptotically stable positive steady state and one unstable positive steady state. 
\item If {$a_+>1$ and $a_->0$}, then there exist {rate} parameter values such that the network has two asymptotically stable positive steady states and one unstable positive steady state.
\end{enumerate}
\end{theorem}

\begin{proof}
 Part (i)  follows immediately from Propositions \ref{prop:stability} and \ref{prop:embedded} by considering any embedded network with one species. We now prove (ii).  If there exists $i$ such that $a_i>1$, then the statement is a consequence of Propositions \ref{prop:stability} and \ref{prop:embedded}, by considering 
 the embedded network on $X_i$. If $a_i\leq 1$ for all $i$, then we proceed as follows. 
Assume for simplicity $a_1=a_2=1$ and recall that $b_1>a_1$, $b_2>a_2$.  Consider the embedded network $N$ on $\{X_1,X_2\}$ and the network $\tilde{N}$ with one species $Y$ as constructed in Lemma \ref{lemma:extending} with $\tilde{a}=a_1+a_2=2$ and $\tilde{b}=b_1+b_2$. 
Choose reaction rate constants $\ell, \; \tilde{\cc}_{1}, \; \tilde{\k}_{1}$ such that $\tilde{N}$ admits two positive non-degenerate steady states. 
Let $y^*$ be one such steady state. Then $(y^*,y^*)$  is a positive steady state of $N$ for the choice of reaction rate constants given in Lemma \ref{lemma:extending}. We show now that if $y^*$ is asymptotically stable (respectively unstable), then so is $(y^*,y^*)$.
Consider 
$M(y^*)= 2(\tilde{b}-2)\ell  y^* - \tilde{\cc}_{1}.  $
If $M(y^*)<0$, then $y^*$ is asymptotically stable and if $M(y^*)>0$, then $y^*$ is unstable. 
Now consider the equations $ f_i(x)= (b_i-1) \ell \, x_1 x_2 - \cc_i x_i + \k_i$, for $i=1,2$, defining the ODE system for $N$. 
By letting $\lambda= \frac{(b_1-1)(b_2-1)}{(\tilde{b}-2)^2} >0$,  the determinant and trace of the Jacobian matrix $J$ of $(f_1,f_2)$ evaluated at $(y^*,y^*)$ are
\begin{align*}
&\det(J(y^*,y^*)) =-(b_1-1)\ell\cc_{2} y^* -(b_2-1)\ell \cc_{1} y^* + \cc_{1} \cc_{2}=-\lambda \tilde{\cc}_{1}M(y^*), \\
&\tr(J(y^*,y^*)) =  (b_1-1)\ell y^* + (b_2-1) \ell y^* - \cc_{1}- \cc_{2}  = (\tilde{b}-2) \ell y^* -  \tilde{\cc}_{1}<  M(y^*),
\end{align*}
where we use that $\tilde{b}>\tilde{a}=2$. 
It follows that if $M(y^*)>0$, then $y^*$ is unstable and so is $(y^*,y^*)$. If $M(y^*)<0$, then $y^*$ is asymptotically stable and 
$(y^*,y^*)$ is as well, since $\det(J(y^*,y^*))>0$ and $\tr(J(y^*,y^*))<0$.
 This concludes the proof of (ii).

Now consider (iii). 
Let $j$ such that $a_j>0$ and $\sign(j)=-1$. 
If   $a_{i}>1$ for some $i$ such that $\sign(i)=1$, then we consider the embedded network on $\{X_{i},X_j\}$. The result follows from Theorem \ref{thm:atMost3Sol} and Propositions \ref{prop:stability3} and \ref{prop:embedded}. If $a_i\leq 1$ for all $i$ such that $\sign(i)=1$, we proceed similarly to the proof of (ii). Assume $a_1=a_2=1$, $a_1<b_1$, $a_2<b_2$ and $j=3$. 
Consider the embedded network $N$ on $\{X_1,X_2,X_3\}$, and let $\tilde{N}$ be the reaction network on the species $Y_1,Y_3$ defined as in Lemma \ref{lemma:extending} such that $\tilde{a}=2$. By Proposition \ref{prop:embedded}, it is enough to show that the statement holds for $N$. 
 By Theorem \ref{thm:atMost3Sol} and Proposition \ref{prop:stability3}, we can choose reaction rate constants $\ell, \; \tilde{\k}_{1},\; \tilde{\k}_{3}$ and $\tilde{\cc}_{1}<\tilde{\cc}_{3}$ such that $\tilde{N}$ admits three positive non-degenerate steady states, two of which are asymptotically stable and the other unstable. By the proof of Proposition \ref{prop:stability3}, the Jacobian matrix $\tilde{J}$  is such that $\det(\tilde{J})$ is negative when evaluated at the unstable steady state, and is also such that $\det(\tilde{J})>0$ and $\tr(\tilde{J})<0$ when evaluated at the asymptotically stable steady states. In particular, we have 
\begin{align*}
&\det(\tilde{J}(y_1,y_3)) = a_3(a_3-b_3) \ell  \tilde{\cc}_{1}   y_1^2y_3^{a_3-1} -2 \ell (\tilde{b}-2)\tilde{\cc}_{3} y_1y_3^{a_3} + \tilde{\cc}_{1}\tilde{\cc}_{3}    , \\
&\tr(\tilde{J}(y_1,y_3)) = - a_3(a_3-b_3) \ell  y_1^2y_3^{a_3-1} + 2\ell(\tilde{b}-2) y_1y_3^{a_3}  - \tilde{\cc}_{1} - \tilde{\cc}_{3}.
\end{align*}
Let $(y_1,y_3)$ be a steady state of $\tilde{N}$ such that $(y_1,y_1,y_3)$ is a steady state of $N$ for the choice of reaction rate constants given in Lemma \ref{lemma:extending}.
The ODE system for $N$ is defined by the polynomials $f_i(x) =  (b_i-a_i) \ell \, x_1 x_2x_3^{a_3} - \cc_i x_i + \k_{i}$ for $i=1,2,3$ (recall that $a_1=a_2=1$). 
The characteristic polynomial of the Jacobian matrix $J$ of $(f_1,f_2,f_3)$ evaluated at $(y_1,y_1,y_3)$ is 
\begin{align*}
 \chi(z)  & = z^3 + \big( a_3  (a_3-b_3)\ell y_1^2 y_3^{a_3-1} - \ell (\tilde{b}-2) y_1 y_3^{a_3} + \cc_{1} + \cc_{2} + \cc_{3}  \big) z^2 \\ & + \big(a_3(a_3-b_3) \ell (\cc_{1}+\cc_{2}) y_1^2 y_3^{a_3-1} - \ell ((b_2-1) \cc_{1} + (b_1-1) \cc_{2} + (\tilde{b}-2)\cc_{3}) y_1 y_3^{a_3}\\ &   +\cc_{1} \cc_{2} +\cc_{1}\cc_{3} + \cc_{2} \cc_{3} \big)z \\ &   
+a_3 (a_3-b_3)\ell \cc_{1} \cc_{2} y_1^2 y_3^{a_3-1} - \ell  ((b_2-1) \cc_{1} + (b_1-1) \cc_{2} )\cc_{3}  y_1 y_3^{a_3}+ \cc_{1}\cc_{2} \cc_{3}.
\end{align*}

Let $\xi_2\;,\xi_1$ and $\xi_0$ be the coefficients of the degree two, one and zero terms of $ \chi(z) $, respectively. 
Since $\xi_0$ has the opposite sign as that of the product of the eigenvalues of $J$, then if $\xi_0<0$ it follows that the steady state is unstable. Otherwise we use the Routh-Hurwitz criterion for polynomials of degree $3$; that is if $\xi_2>0$, $\xi_0>0$ and $\xi_2\xi_1>\xi_0$, then all eigenvalues have negative real part and hence the steady state is asymptotically stable. 

Let $\lambda= \frac{(b_1-1)(b_2-1)}{(\tilde{b}-2)^2} >0$ as above and note that since $\tilde{b}=b_1+b_2$, we have  $\lambda <\frac{1}{2}$. Using the definition of $\k_{i},\cc_i$  in Lemma \ref{lemma:extending} we have\small
\begin{align}
\xi_0 &=\lambda  \tilde{\cc}_{1}  \Big(  a_3(a_3-b_3) \ell    \tilde{\cc}_{1} y_1^2 y_3^{a_3-1} 
-2 \ell  (\tilde{b}-2) \tilde{\cc}_{3} y_1 y_3^{a_3} + \tilde{\cc}_{1} \tilde{\cc}_{3} \Big) = \lambda  \tilde{\cc}_{1} \det(\tilde{J}(y_1,y_3)).\label{eq:xi0}
\end{align}\normalsize
Thus, if $\det(\tilde{J}(y_1,y_3))<0$, then $\xi_0<0$ and $(y_1,y_1,y_3)$ is unstable. 
Assume $\det(\tilde{J}(y_1,y_3))>0$ so that $\xi_0>0$, and also assume $\tr(\tilde{J}(y_1,y_3)) <0$. 
It follows from the latter inequality that 
\begin{equation}\label{eq:boundB2}
\ell  (\tilde{b}-2)y_1y_3^{a_3} <  \tfrac{1}{2} \big(  a_3(a_3-b_3) \ell    y_1^2y_3^{a_3-1}   + \tilde{\cc}_{1}+\tilde{\cc}_{3}  \big).
\end{equation}
Using \eqref{eq:boundB2}, $a_3>b_3$, and  $\lambda<\tfrac{1}{2}$, 
we have
\begin{align}
\xi_2 & = a_3  (a_3-b_3)\ell y_1^2 y_3^{a_3-1} - \ell (\tilde{b}-2) y_1 y_3^{a_3} + \tilde{\cc}_{1} + \tilde{\cc}_{3} \nonumber \\ 
& \qquad\qquad  >  \tfrac{1}{2} \big(a_3(a_3-b_3) \ell    y_1^2y_3^{a_3-1}   + \tilde{\cc}_{1}+\tilde{\cc}_{3} \big)>\lambda \tilde{\cc}_{1} >0. \label{eq:ineqxi2}
\end{align}
All that remains is to show $\xi_2\xi_1>\xi_0$. We first find that
\begin{align*}
\xi_1 &= a_3 (a_3-b_3)\ell  \tilde{\cc}_{1}y_1^2 y_3^{a_3-1} - \ell (\tilde{b}-2)   (2\lambda  \tilde{\cc}_{1} + \tilde{\cc}_{3}) y_1 y_3^{a_3}   + \lambda (\tilde{\cc}_{1})^2  + \tilde{\cc}_{1} \tilde{\cc}_{3}.
\end{align*}
Since $2\lambda  \tilde{\cc}_{1}  < \tilde{\cc}_{1}  <   \tilde{\cc}_{3}$, we have 
that $ 2\lambda  \tilde{\cc}_{1} + \tilde{\cc}_{3} < 2  \tilde{\cc}_{3}$.
It follows that $\xi_1> \det(\tilde{J}(y_1,y_3))$, and in particular $\xi_1>0$.
Combining this inequality with \eqref{eq:xi0} and \eqref{eq:ineqxi2}, we find 
\begin{align*}
\xi_2\xi_1 - \xi_0 &> \lambda \tilde{\cc}_{1} \det(\tilde{J}(y_1,y_3)) - \lambda  \tilde{\cc}_{1} \det(\tilde{J}(y_1,y_3))=0.
\end{align*}
This shows that the steady state $(y_1,y_1,y_3)$ is asymptotically stable, proving (iii).
\end{proof}

\section{Networks with one non-flow reversible reaction}\label{sec:reversible}
In this section we apply Gale duality to study a modified version of the family of networks \eqref{eq:myreaction} analyzed in \S\ref{sec:OneVarGaleSys}. These networks were originally introduced in \cite{Joshi}. In particular we consider the following reaction networks with $n$ species $X_1,\dots, X_n$:
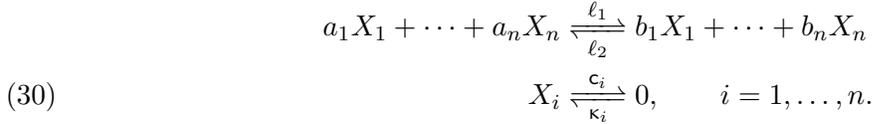
\begin{align*}
 a_1X_1+\cdots +a_nX_n &  \ce{<=>[\ell_{1}][\ell_{2}]} b_1X_1+\cdots+b_nX_n \\ 
X_i &\ce{<=>[\cc_{i}][\k_{i}]} 0,\qquad i=1,\dots,n. \numberthis \label{eq:myreaction2}
\end{align*}
The system of equations describing the steady states of the network \eqref{eq:myreaction2} is:
\begin{equation}
(b_i-a_i)\ell_1x^a-(b_i-a_i)\ell_2x^b-\cc_{i}x_i+\k_{i} =0, \;\; {\rm for\;}i=1,\dots,n. \label{eq:Sys2}
\end{equation}
As in \S\ref{sec:OneVarGaleSys} we let $\sgn(i)=\sgn(b_i-a_i)$.
In the remaining sections we will assume that $\sign(i)\neq 0$ for all $i$. This assumption is without loss of generality, see Remark \ref{remark:a_i_pos}. 
Similar to \eqref{eq:signs0}, we let 
\begin{equation}\label{eq:signs}
a_+=\sum_{\sgn(i)=1} a_i, \;\; b_+=\sum_{\sgn(i)=1} b_i, \;\; a_-=\sum_{\sgn(i)=-1} a_i, \;\; {\rm and}\;\; b_-=\sum_{\sgn(i)=-1} b_i.
\end{equation}

{In \cite{Joshi} it was shown that the system \eqref{eq:Sys2} admits at least two positive solutions for some $\ell_1,\ell_2,\cc_i$ and $\k_i$ if and only if $ a_+>1$ or $b_->1.$
} Note that the network \eqref{eq:myreaction2} is obtained by making an irreversible non-flow reaction of the network  \eqref{eq:myreaction} reversible. {This construction applies to either the reaction with label $\ell_1$ or the one with label $\ell_2$.} Hence, by Proposition~\ref{prop:subnets} and the results in \S \ref{sec:OneVarGaleSys} on the steady states of the network \eqref{eq:myreaction}, we have:
\begin{itemize}
\item If {either $a_+>1$ and $a_-=0$ or $b_->1$ and $b_+=0$}, then the network \eqref{eq:myreaction2} admits at least two positive non-degenerate steady states. 
\item If {either $a_+>1$ and $a_->0$ or $b_->1$ and $b_+>0$}, then the network \eqref{eq:myreaction2} admits at least three positive non-degenerate steady states. 
\end{itemize}
In all other cases there exist values of the reaction rate constants such that the network  \eqref{eq:myreaction2} has one positive non-degenerate steady state.

\begin{conjecture}\label{conjecture:atMost3}
Based on the analysis presented below we conjecture that the maximum possible number of positive steady states of the reaction network \eqref{eq:myreaction2} is three. 
\end{conjecture}

Again, we will study the number of positive steady states by constructing a Gale dual system. 
{Since, by assumption, $a_i\neq b_i$ for all $i$, we let}
$$\alpha:=\frac{\ell_2}{\ell_1}, \quad \; c_{i}:=\frac{\cc_i}{|b_i-a_i|\ell_1}, \quad k_i:=\frac{\k_i}{|b_i-a_i|{\ell_1}}.$$ We may rewrite the system \eqref{eq:Sys2} as
\begin{equation}
{\rm sgn}(i)\left(x^a-\alpha\, x^b\right)-c_{i}x_i+k_{i} =0, \;\; {\rm for\;}i=1,\dots,n.\label{eq:Sys2Simp}
\end{equation}
Using the monomial vector $x^{W}=[
 x_1 ,\dots,x_n , x^a , x^b ,1
]^T$ the system of equations \eqref{eq:Sys2Simp} can also be expressed in the form $C \cdot x^{{W}}=0,$ where\small \begin{equation}
C=\begin{bmatrix}
-c_{1}&  \cdots & 0 & {\rm sgn}(1) & -{\rm sgn}(1)\alpha& k_{1}\\
\vdots  & \ddots & \vdots& \vdots & \vdots & \vdots   \\
  0  &\cdots &-c_{n} & {\rm sgn}(n)& -{\rm sgn}(n)\alpha& k_{n}\\
\end{bmatrix}\; {\rm and}\;\;\;W = \begin{bmatrix}
1 & \dots & 0 & a_1& b_1 &0 \\ 
\vdots   & \ddots & \vdots & \vdots  & \vdots & \vdots\\
0 & \dots & 1 & a_n& b_n&0
\end{bmatrix}\hspace{-0.1cm}.\label{eq:System2CWMat}
\end{equation} \normalsize
We now compute a Gale dual system to \eqref{eq:Sys2Simp} using Proposition \ref{Prop:GaleDual2}. We have that $l=2$ and hence any Gale dual system depends on two variables, $y_1$ and $y_2$. One can check that the following choices of Gale dual matrices $ D\in \RR^{(n+3)\times 3}$ and $Q\in \ZZ^{(n+3)\times 3}$ satisfy the requirements stated in Proposition \ref{Prop:GaleDual2}:
\begin{equation}
D=\begin{bmatrix}
 \frac{{\rm sgn}(1)}{c_{1}}& 0 &\frac{k_{1}}{c_{1}} \\
\vdots & \vdots & \vdots\\
 \frac{{\rm sgn}(n)}{c_{n}}&  0 &\frac{k_{n}}{c_{n}}\\
 1 &  1&0 \\
 0 & \frac{1}{\alpha}&0\\
0 & 0 & 1
\end{bmatrix}, \;\;\; Q=\begin{bmatrix}
a_1& b_1& 0\\
\vdots & \vdots & \vdots\\
a_n & b_n & 0 \\
-1 & 0 & 0 \\
0 & -1 & 0 \\
0& 0&1
\end{bmatrix} .
\end{equation}In the notation of \S\ref{sec:galedual}, we have
$$ d_i(y_1,y_2)= \tfrac{k_i}{c_i}+\tfrac{\sign(i)}{c_i}y_1,\quad i=1,\dots,n,\quad d_{n+1}(y_1,y_2)=y_1+y_2, \; {\rm and\;\;} d_{n+2}(y_1,y_2)=\tfrac{1}{\alpha}y_2.   $$
Hence the Gale dual system to \eqref{eq:Sys2Simp} in $\RR[y_1,y_2]$ is given by 
\begin{equation}
g_1(y_1,y_2)=g_2(y_1,y_2)=0,\; {\rm for \;} y_1>0,\; y_2>0\;\; {\rm and\;} (k_i + \sgn(i)y_1)>0\label{GaleSysg3}
\end{equation}where
$$
g_1(y_1,y_2)=\prod_{i=1}^n \left({\rm sgn}(i)y_1 +k_i \right)^{a_i}-c^a (y_1 +y_2),\;{\rm and}\;\;g_2(y_1,y_2)=\prod_{i=1}^n \left({\rm sgn}(i)y_1 +k_i \right)^{b_i}- \tfrac{c^b}{\alpha}y_2.
$$\normalsize
From the relation $g_2(y_1,y_2)=0$  we have that
\begin{equation}
 y_2 = \frac{\alpha}{c^b} \prod_{i=1}^n \left({\rm sgn}(i)y_1 +k_i \right)^{b_i}. \label{eq:y2InTermsy1}
\end{equation}
After substituting $y_2$ into $g_1$ and scaling by a constant we arrive at the following expression
\begin{equation}
 \widetilde{g}_1(y_1)= c^{-a}\prod_{i=1}^n \left({\rm sgn}(i)y_1 +k_i \right)^{a_i} - \alpha c^{-b} \prod_{i=1}^n \left({\rm sgn}(i)y_1 +k_i \right)^{b_i} 
-  y_1.  \label{eq:TwoVarSingvarSys}
\end{equation} 
Hence, we will study the positive steady states of the network using the system 
\begin{equation}
 \widetilde{g}_1(y_1)=0 ,\;\; {\rm for \;} y_1>0 {\rm \;and \;} (k_i + \sgn(i)y_1)>0.\label{eq:simpGale_y1}
\end{equation} Note that from \eqref{eq:y2InTermsy1} we see that the stated constraints on $y_1$ ensure that the constraint $y_2>0$ appearing in \eqref{GaleSysg3} is always satisfied. As in \S\ref{sec:OneVarGaleSys} we will study the number of solutions of \eqref{eq:simpGale_y1} by considering the different cases arising from the values of {$a_+,a_-,b_+$, and $b_-$ in \eqref{eq:signs}.} {Note that by symmetry, a system  analogous to \eqref{eq:simpGale_y1} is obtained by reversing the roles of $a$ and $b$, and replacing $\alpha$ by $\alpha^{-1}$.}

\subsection{Number of steady states}\label{subsec:numSteadyState_2}
As in \eqref{eq:kminusDef} let $k_-$ be the smallest value of $k_i$ for which $\sgn(i)$ is negative {(recall that we set $k_-=+\infty$ if there is no negative sign)}. With $\widetilde{g}_1$ as in \eqref{eq:TwoVarSingvarSys} 
  we study the solutions to $\widetilde{g}_1(y_1)= 0$ for $y_1\in(0,k_-)$. 
  Our analysis of the number of positive steady state of the network \eqref{eq:myreaction2} focuses on the following two cases:
\begin{enumerate}[(i)]
\item {either $a_+>1$ and $a_-=0$, or $b_->1$ and $b_+=0$}, 
\item {either $a_+>1$ and $a_->0$, or $b_->1$ and $b_+>0$}.
\end{enumerate} 
{Case (i) implies that either $\sign(i)=1$ for all $i$, or $\sign(i)=-1$ for all $i$ (for example, $a_-=0$ implies that $a_i>b_i$ is not possible). In case (ii) both signs occur.}

\subsubsection{Case (i)}\label{subsubsec:allSignsEqual}
Due to symmetry, we assume, without loss of generality, that {$a_+>1$ and $a_-=0$}. 
In this case the constraints $ ({\rm sgn}(i)y_1 +k_i) >0,\; y_1>0$ are simply given by $y_1>0$, so we want to determine the positive roots of $\widetilde{g}_1$. 
\begin{remark}\label{remark:DescarteFuncSeq}
A finite sequence of real valued functions $f_1(y),\dots, f_\ell(y)$ satisfies Descartes' rule of signs on an interval $(a,b)\subset \RR$ if the number of zeros of the function $\mathfrak{c}_1f_1(y)+\cdots+\mathfrak{c}_\ell f_\ell(y)$ in $(a,b)$ (counted with multiplicity) is less than or equal to the number of sign changes in the sequence $\mathfrak{c}_1,\dots, \mathfrak{c}_\ell$ for any choice of real constants $\mathfrak{c}_i$. For more see \cite[Part~5, 87--90]{polya1997problems}. 
\end{remark}{
\begin{remark}\label{remark:SpecificDecFuncSeq}
Let $\omega_0,\dots, \omega_\ell$ be arbitrary real numbers. A classical result of Runge (see, for example, \cite[pg.~36]{dimitrov2009descartes}) states that the sequence of polynomials \begin{equation}
1,\ y-\omega_0,\ (y-\omega_0)(y-\omega_1), \dots, \ (y-\omega_0)(y-\omega_1)\cdots (y-\omega_\ell) \label{eq:DescartSeq1}
\end{equation} satisfies Descartes' rule of signs (in the sense of Remark \ref{remark:DescarteFuncSeq}) for $ y>\max(\omega_i\;|\;i=0,\dots,\ell)$; this also holds for any subsequence of this sequence.
\end{remark}}
We now show that the maximum number of positive roots of $\widetilde{g}_1$ is three. 
\begin{proposition}\label{propn:AtMost3Case1Sec4}
{If $a_+>1$ and $a_-=0$}, the equation 
$$
\widetilde{g}_1(y_1)=c^{-a}\prod_{i=1}^n \left(y_1 +k_i \right)^{a_i} - \alpha c^{-b} \prod_{i=1}^n \left(y_1 +k_i \right)^{b_i} 
-  y_1=0
$$ 
has at most three positive solutions (counted with multiplicity), and there exist values of the parameters for which $\widetilde{g}_1$ has three distinct positive roots.

Hence the network \eqref{eq:myreaction2} admits three positive non-degenerate steady states when $a_+>1$ and $a_-=0$ (by symmetry, this also holds when $b_->1$ and $b_+=0$).
\end{proposition}
\begin{proof}  We have that $\sign(i)=1$ for all $i$. Set $k_+=\min(k_i\;|\; i=1,\dots,n)$. 
Let $1\leq \nu\leq n$ be an integer {such that $a_\nu>0$} and consider the sequence of functions \begin{equation}
1, \ y_1+k_{\nu}, \ \gamma(y_1)=\prod_{i=1}^n \left(y_1 +k_i \right)^{a_i}, \  \omega(y_1)=\prod_{i=1}^n \left(y_1 +k_i \right)^{b_i}.\label{eq:seqenceFuncs}
\end{equation} 
{By Remark \ref{remark:SpecificDecFuncSeq} it follows that the sequence of functions \eqref{eq:seqenceFuncs} satisfies Descartes' rule of signs on the open interval $y_1>-k_+$. To see this take $\omega_0=-k_\nu$, $\omega_1=\cdots=\omega_{a_1}=-k_1$, $\omega_{a_1+1}=\cdots=\omega_{a_1+a_2}=-k_2$, and so on (omitting one $\omega_j=-k_i$ when $i=\nu$) until we reach $\omega_{a_+-1}=-k_n$, then set $\omega_{a_+}=\cdots=\omega_{a_++(b_1-a_1)}=-k_1$ and so on until we reach {$\omega_{b_+-1}=-k_n$}. With this choice of $\omega$ the sequence of functions \eqref{eq:seqenceFuncs} is a subsequence of that in \eqref{eq:DescartSeq1} (recall that $b_i>a_i$ for all $i$).}
Rewrite $\widetilde{g}_1$ as $$
\widetilde{g}_1(y_1)=k_{\nu}-  (y_1+k_{\nu} ) +c^{-a}\gamma(y_1) - \alpha c^{-b} \omega(y_1).
$$The sign sequence for $\widetilde{g}_1$ with respect to the sequence of functions \eqref{eq:seqenceFuncs} is $+,-,+,-$. Hence, by Remark \ref{remark:DescarteFuncSeq} we have that $\widetilde{g}_1$ has at most three roots in the interval $(-k_+,\infty)$. In particular $\widetilde{g}_1$ has at most three positive roots. 

{To show that three solutions are possible, set $x_1=\dots=x_n$, $k_{1}=\dots=k_{n}$ and $c_{1}=\dots=c_{n}$ in \eqref{eq:Sys2Simp} and recall that $\sgn(i)=1$ for all $i$, meaning $b_+>a_+$
(also note that $b_i-a_i$ is a fixed value, so we choose each $\k_i$ and $\cc_i$ appropriately to yield $k_{1}=\dots=k_{n}$ and $c_{1}=\dots=c_{n}$). Then  for all $i$,  the equations in \eqref{eq:Sys2Simp} have the form 
$$
{\rm sgn}(i)\left(x_1^{a_1}\cdots x_n^{a_n}-\alpha\, x_1^{b_1}\cdots x_n^{b_n}\right)-c_{i}x_i+k_{i} =x_1^{{a_+}}-\alpha\, x_1^{{b_+}} - c_1 x_1 + k_{1}=0$$
for $i=1,\dots, n$. In particular, all steady state equations \eqref{eq:Sys2Simp} are the same for all $i$ and are equal to 
\begin{equation}
p(x_1)=  k_{1}- c_1 x_1 + x_1^{a_+}-\alpha x_1^{b_+} =0.\label{eq:pxInSec3_1_proof}
\end{equation} }%
 The polynomial $\frac{-p(x_1)}{\alpha}$ is an arbitrary monic polynomial with real coefficients of fixed sign. Since $b_+>a_+$, its Descartes sign sequence is {$+,-,+,-$} {(starting with the constant term)}, hence there must exist values of the coefficients for which Descartes bound of three positive solutions is achieved.
\end{proof}

Let 
\begin{equation}
v(y_1)={\prod_{i=1}^n \left(y_1 +k_i \right)^{a_i} -c^a\cdot y_1}, \;\;\;w(y_1)=\alpha c^{a-b} \prod_{i=1}^n \left(y_1 +k_i \right)^{b_i}.\label{eq:v_w}
\end{equation} 
To find the solutions to \eqref{eq:TwoVarSingvarSys} we seek the points $y_1>0$ where $v$ and $w$ intersect. Note that $$
v(0)=k^a, \;\;\;\; v'(0)={k^a} \sum_{j=1}^n \frac{a_j}{k_{j} }-c^a,\;\;\;\; {\rm and}\;\;
w(0)=\frac{\alpha k^b}{c^{b-a}}.
$$

 \begin{wrapfigure}{r}{0.48\textwidth}\vspace{-0.1in}
  \begin{center}
\begin{tikzpicture}[very thick,scale=0.99, every node/.style={scale=0.83}]
\begin{axis}[
    axis lines = middle,
        ymax=10000,
        xmax=7,
        xmin=-0.05,
        ymin= -1000,
    ticks=none,
]
\addplot [
    domain=-0.05:5.79, 
    samples=150, 
    color=oran,smooth, thick
]
{22*(x+5)^2*(x+1)-1518*x} node[right,scale=1.2] {$v(y_1)$};
\addplot [
    domain=-0.05:5.89,
    samples=150, 
    color=MyCiteColor,smooth, thick]
    {(1/69)*(x+5)^4*(x+1)^2} node[left,scale=1.2] {${w(y_1)}$};
     \addplot[
    color=MyRed,
    only marks,
    mark=|,
    mark size=1.5pt,
    ]
    coordinates {
    (1.311061508,0)
    } node[below,scale=1.2] {$\xi$};
         \addplot[
    color=MyRed,
    only marks,
    mark=|,
    mark size=1.5pt,
    ]
    coordinates {
    (4.1,0)
    } node[below,scale=1.2] {$\sigma$};
    \addplot[
    color=Red2,
    only marks,
    mark=*,
    mark size=1.1pt,
    ]
    coordinates {
    (.9409700891,68.017092)
    };
     \addplot[
    color=Red2,
    only marks,
    mark=*,
    mark size=1.1pt,
    ]
    coordinates {
    (2.595817541,623.798461)
    };
     \addplot[
    color=Red2,
    only marks,
    mark=*,
    mark size=1.1pt,
    ]
    coordinates {
    (5.478251216,7331.987334)
    };

\end{axis}
\end{tikzpicture}  \end{center}\vspace{-0.1in}
  \caption{\footnotesize The situation described in Proposition \ref{prop:v_w_intersect} (iii);  $c_i$ and $k_i$ are chosen so that $v(0)>w(0)$ and $v'(0)<0$, and $\alpha$ is chosen so that $v(\xi)<w(\xi)$ and $v(\sigma)>w(\sigma)$ where $\xi$ is a positive root of $v'$ and $\sigma>\xi$.\label{fig:3SolsSec4}}\vspace{-0.1in}
\end{wrapfigure}
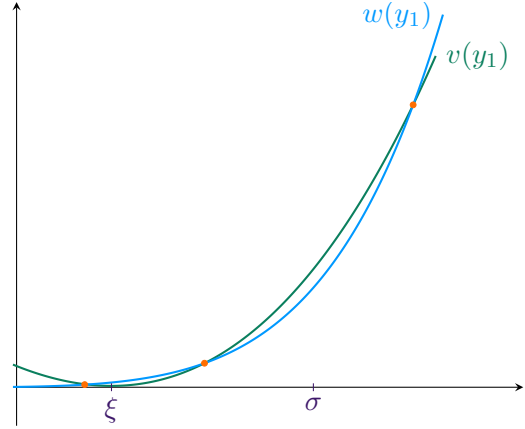 

\begin{proposition}\label{prop:v_w_intersect}
{Suppose that $a_i<b_i$ for all $i$}. Let $v$ and $w$ be as in \eqref{eq:v_w} and consider the solutions to $\widetilde{g}_1(y_1)=0$ for $y_1>0$ with $\widetilde{g}_1$ as in \eqref{eq:TwoVarSingvarSys}. 

\begin{enumerate}[(i)]
\item If {$v(0)\leq w(0)$}, then $\widetilde{g}_1$ has at most two positive roots.
\item If $v(0)>w(0)$, then $\widetilde{g}_1$ has at least one positive root.
\item Suppose that $v(0)>w(0)$ and $v'(0)<0$. Let $\xi$ be the positive root of $v'$ and let $\sigma>\xi$. Then $\widetilde{g}_1$ has exactly three positive roots if $v(\xi)<w(\xi)$ and $v(\sigma)>w(\sigma)$.
\end{enumerate}
\end{proposition}
\noindent\textit{Proof.} Let $k_+=\min(k_i\;|\; i=1,\dots,n)$. By the proof of Proposition \ref{propn:AtMost3Case1Sec4} we know that $v$ and $w$ intersect at most three times in the interval $(-k_+,\infty)$. Also note that $v(-k_+)>0$ and $w(-k_+)=0$. Hence if $v(0)\leq w(0)$, then $v$ and $w$ intersect at least once in $(-k_+,0]$, and hence at most twice in $(0,\infty)$; this proves (i). 

Item (ii) follows immediately from the fact that $w(y_1)$ is larger than $v(y_1)$ for sufficiently large $y_1$. 

Now consider (iii). Suppose that $v(0)>w(0)$ and $v'(0)<0$. Since $v'(\theta)>0$ for $\theta$ large, and since $v'$ has at most one positive real root by Descartes' rule of signs, then $v'$ must have exactly one positive real root $\xi$. If $v(\xi)<w(\xi)$, then $v$ and $w$ intersect at least once in $(0,\xi)$. Similarly if $v(\sigma)>w(\sigma)$ for some $\sigma>\xi$, then $v$ and $w$ must intersect at least once in $(\xi,\sigma)$. Finally we know that for $\zeta$ sufficiently large $v(\zeta)<w(\zeta)$, hence $v$ and $w$ must intersect at least once in the interval $(\sigma,\infty)$. Taken together this implies that $v$ and $w$ intersect exactly once in each of the intervals $(0,\xi)$, $(\xi, \sigma)$ and $(\sigma, \infty)$; this proves (iii). Note that the condition $v'(0)<0$ is necessary since $v'$ is strictly increasing for all positive $y_1$, hence if $v'(0)>0$, then $v'$ has no positive roots. The situation described by (iii) is illustrated in Figure \ref{fig:3SolsSec4}. \qed
\color{black}

\subsubsection{Case (ii)} We now consider the case where either $a_+>1$ and $a_->0$ or $b_->1$ and $b_+>0$ (note we must have $n\geq 2$). 
By computing a cylindrical algebraic decomposition we have found that the maximum number of roots of $\widetilde{g}_1$ in $ (0,k_-)$ is three for the following combinations of vectors $a$ and $b$:

\begin{center}
\begin{tabu}{c|[1pt]ccccccc|cc|}
&  \multicolumn{7}{c|}{$n=2$}  & \multicolumn{2}{c|}{$n=3$}  \\[2pt] \tabucline[1pt]{-} 
$[a_1,\dots,a_n]$ & $ [2,1] $ & $[2,2]$    &  $[2,2]$    &  $[2,2]$    &  $[2,2]$    &  $[2,3]$ &       $[3,1]$   & $[2,1,1]$ & $[2,3,1]$   \\
$[b_1,\dots,b_n]$ &    $[3,0]$ &  $[4,0]$ &  $[4,1]$ &  $[5,1]$ &  $[6,1]$ &  $[3,2]$ &  $[4,0]$  & $[3,0,0]$ & $[3,2,0]$
\end{tabu}
\end{center}

\smallskip
\noindent
We have seen in \S\ref{sec:OneVarGaleSys} that for networks with one non-flow irreversible reaction, increasing $n$ does not change the possible number of positive steady states of the network. In light of this, along with the results of \S\ref{subsubsec:allSignsEqual} and the investigations using cylindrical algebraic decomposition, we conjecture that $\widetilde{g}_1$ admits up to three roots in $(0,k_-)$, see Conjecture \ref{conjecture:atMost3}.

We now give a result which uses Remark \ref{remark:DescarteFuncSeq} to bound the number of roots of $\widetilde{g}_1$ in $(0,k_-)$. 

\begin{proposition}\label{prop:mixedSignBound}
{Let $a_+>0$,} and {let $\widetilde{g}_1$ is as in \eqref{eq:TwoVarSingvarSys}. The number of roots of $\widetilde{g}_1$ in $ (0,k_-)$ is at most $\max(a_-,  b_+ + b_- - a_+ )+2$}. 
\end{proposition}
\begin{proof}
Since $a_+>0$, there exists an index $\nu$ such that $a_\nu>0$.
Let $\Omega=\max(a_-,  b_+ + b_- - a_+ )$. 
By Remark \ref{remark:SpecificDecFuncSeq} it follows that the sequence of $\Omega+3$ functions 
\begin{equation}
1,\ y_1+k_\nu,\  \prod_{\sign(i)=1} \left(y_1 +k_{i} \right)^{a_i}, \ y_1\prod_{\sign(i)=1} \left(y_1 +k_{i} \right)^{a_i}, \  \dots,  y_1^{\Omega}\prod_{\sign(i)=1} \left(y_1 +k_i \right)^{a_i}\label{eq:MixedSignDescarteFuncs}
\end{equation}
satisfies Descartes' rule of signs on for $y_1>0$. Let $$
\gamma(y_1)=c^{-a}\prod_{\sign(i)=-1} \left(-y_1 +k_i \right)^{a_i} - \alpha c^{-b} \prod_{\sign(i)=1} \left(y_1 +k_i \right)^{b_i-a_i}\prod_{\sign(i)=-1} \left(-y_1 +k_i \right)^{b_i} ,
$$ and note that $\deg(\gamma)=\Omega$. With this notation $\widetilde{g}_1(y_1)=\prod_{\sign(i)=1} \left(y_1 +k_i \right)^{a_i} \gamma(y_1)-  (y_1+k_\nu)+k_\nu$. Expanding $\gamma$ we see that  
the polynomial $ \widetilde{g}_1$ has at most $\Omega+2$ sign changes when written in terms of the sequence of $\Omega+3$ functions \eqref{eq:MixedSignDescarteFuncs}. 
\end{proof}

{By symmetry, we conclude that network \eqref{eq:myreaction2} admits at most 
$$ \min ( \max(a_-,  b_+ + b_- - a_+ ),\max(b_+,  a_++a_- - b_- )  )+2 $$
positive non-degenerate steady states.}

\begin{remark}
By \cite[Theorem~1]{bates2007bounds} the number of non-zero real solutions to the system \eqref{eq:Sys2Simp} is no more than $n^2 (e^4+3)$. Alternatively, by the Bernstein-Kushnirenko theorem, the number of non-zero real solutions to \eqref{eq:Sys2Simp} is no more that $n! {\rm vol}({\rm conv}(W))$, where ${\rm vol}({\rm conv}(W))$ denotes the Euclidean volume of the convex hull of the columns of the matrix $W$ in \eqref{eq:System2CWMat}. These bounds are both much larger than the bound of three for the case (i) explored in \S\ref{subsubsec:allSignsEqual} and are also often larger than the bound of Proposition \ref{prop:mixedSignBound}. For example, consider the case where $n=4$, $a=[2,3,3,6]$, and $b=[1,2,4,7]$. In this case Proposition \ref{prop:mixedSignBound} tells us that there are at most $ \max(a_-,  b_+ + b_- - a_+ )+2=5+2=7$ positive solutions to \eqref{eq:Sys2Simp}, on the other hand $4^2 (e^4+3)\simeq 876.57$, and $4! {\rm vol}({\rm conv}(W))=40$.
\end{remark}

\subsection{Stability}\label{subsec:stab_steady_2}
In this section we show that if a network of the type \eqref{eq:myreaction2} admits three positive steady states, then it displays bistability. We start with a lemma on bistability. 

\begin{lemma}\label{prop:stability4}
Suppose that either $n=1$ and $1<a_1<b_1$, or $n=2$ and $a_1=1,a_2=1$ $\sign(1)=\sign(2)=1$. 
If  the network \eqref{eq:myreaction2} has three positive steady states, then two of them are asymptotically stable  and one is unstable.
\end{lemma}
\begin{proof}
First consider the case $n=1$. The ODE system of the network \eqref{eq:myreaction2} is 
\begin{align*}
\frac{dx_1}{dt} &=f(x_1),\qquad\textrm{with } f(x_1)=  (b_1-a_1)\ell_1 x_1^{a_1} + (a_1-b_1) \ell_2 x_1^{b_1} - \cc_{1}x_1+\k_{1}.
\end{align*}
Since  $f(0)>0$, if the network has three positive steady states (which must be of multiplicity one), then the smallest and largest steady states satisfy $f'(x_1^*)<0$ and are asymptotically stable, and the middle one satisfies $f'(x_1^*)>0$ and is unstable. 

Now consider the case where $n=2$, $a_1=1$ and $ a_2=1$. 
The ODE system is of the form
\begin{align*}
\frac{dx_i}{dt} &=\alpha_i\ell_1  x_1x_2-\alpha_i\ell_2  x_1^{b_1}x_2^{b_2} - \cc_{i}x_i+\k_{i}, \;\;\; {\rm for \;}i=1,2
\end{align*}
where $\alpha_i=b_i-1>0$. 
The associated Jacobian matrix is
$$J(x) =   \begin{bmatrix}
 \alpha_1 \ell_1 x_2 - b_1\alpha_1 \ell_2  x_1^{b_1-1}x_2^{b_2}  - \cc_{1}  &   \alpha_1 \ell_1 x_1 - b_2\alpha_1\ell_2  x_1^{b_1}x_2^{b_2-1}  \\[6pt]
\alpha_2\ell_1x_2-  b_1\alpha_2 \ell_2 x_1^{b_1-1}x_2^{b_2}  &  \alpha_2\ell_1 x_1   - b_2\alpha_2\ell_2 x_1^{b_1}x_2^{b_2-1}   - \cc_{2}
\end{bmatrix}.$$
We 
have
\begin{align*}
\det(J(x)) &= - \cc_{2} ( \alpha_1 \ell_1 x_2 - b_1\alpha_1\ell_2  x_1^{b_1-1}x_2^{b_2} ) - \cc_{1}( \alpha_2 \ell_1 x_1  - b_2\alpha_2\ell_2 x_1^{b_1}x_2^{b_2-1}  ) + \cc_{1} \cc_{2}, \\
\tr(J(x)) &=   (  \alpha_1 \ell_1 x_2 - b_1\alpha_1\ell_2  x_1^{b_1-1}x_2^{b_2} )  +(\alpha_2 \ell_1 x_1   - b_2\alpha_2\ell_2 x_1^{b_1}x_2^{b_2-1}  ) - \cc_{1}- \cc_{2}.
\end{align*}
A steady state of \eqref{eq:myreaction2} satisfies the equation
$
x_2 =\frac {\alpha_2\cc_{1}x_1 - \alpha_2 \k_{1} + \alpha_1\k_{{2}}}{\alpha_{{1}}\cc_{2}},
$
and $f(x_1)=0$ where \small
\begin{equation}\label{eq:x2A}
 f(x_1)=\alpha_1\ell_1 x_1 \left( \frac {\alpha_2\cc_{1}x_1- \alpha_2 \k_{1} + \alpha_1\k_{{2}}}{\alpha_{{1}}\cc_{2}}  \right)  - \alpha_1 \ell_2 x_1^{b_1}  \left( \frac {\alpha_2\cc_{1}x_1- \alpha_2 \k_1 + \alpha_1\k_{{2}}}{\alpha_{{1}}\cc_{2}}  \right) ^{\hspace{-1mm}b_{{2}}}     -\cc_{1}x_{{1}}+\k_{{1}}.
 \end{equation}\normalsize
If the network has  three positive steady states of multiplicity one, then two of them satisfy $f'(x_1)<0$ and the other satisfies $f'(x_1)>0$. 
Therefore, it is enough to show that if $f'(x_1)<0$, then the steady state is asymptotically stable, and if $f'(x_1)>0$, then the steady state is unstable. Let $x=(x_1,x_2)$ be a positive steady state.
Using \eqref{eq:x2A}, the derivative of $f$ at $x_1$ is  
\begin{align*}
f'(x_1) & =\alpha_1\ell_1 x_2 +  \alpha_2\ell_1  x_1 \tfrac{\cc_{1}}{ \cc_{2}} 
 -b_1\alpha_1\ell_2 x_1^{b_1-1} x_2^{b_2} - b_2 \alpha_2\ell_2 x_1^{b_1} x_2^{b_2-1} \tfrac{\cc_{1}}{ \cc_{2}} 
 -\cc_{1} = \tfrac{-\det(J(x))}{\cc_{2}}.
\end{align*}
 It follows that if $f'(x_1)>0$, then  $\det(J(x))<0$ and the steady state is unstable. If $f'(x_1)<0$, then $\det(J(x))>0$. In order to show that the steady state is asymptotically stable, we need to show that $\tr(J(x))<0$.
From $\tfrac{dx_1}{dt}=\tfrac{dx_2}{dt}=0$ we obtain
$$ \alpha_1\ell_1  x_1x_2 = \alpha_1\ell_2  x_1^{b_1}x_2^{b_2} + \cc_{1}x_1-\k_{1}, \quad 
\alpha_2\ell_1  x_1x_2 =  \alpha_2\ell_2  x_1^{b_1}x_2^{b_2} +  \cc_{2}x_2-\k_{2}, $$
which gives
\begin{align*}
x_1x_2\tr(J(x)) &=     x_2( \alpha_1\ell_2  x_1^{b_1}x_2^{b_2} + \cc_{1}x_1-\k_{1})  - b_1\alpha_1\ell_2  x_1^{b_1}x_2^{b_2+1}   \\ & +  x_1 (  \alpha_2\ell_2  x_1^{b_1}x_2^{b_2} +  \cc_{2}x_2-\k_{2}) - b_2\alpha_2\ell_2 x_1^{b_1+1}x_2^{b_2}   - (\cc_{1}+ \cc_{2})x_1x_2 \\ 
&= -\k_{1}x_2 +(1 - b_1)\alpha_1\ell_2  x_1^{b_1}x_2^{b_2+1}   -\k_{2}x_1 +(1- b_2)\alpha_2\ell_2 x_1^{b_1+1}x_2^{b_2}  <0.\qedhere
\end{align*}
\end{proof}

\begin{theorem}
 If $a_+>1$ or $b_->1$, then  there exists a choice of parameters such that the network \eqref{eq:myreaction2} has two asymptotically stable positive steady states and one unstable positive steady state. 
 \end{theorem}
\begin{proof}
If either $a_+>1$ and $a_->0$ or $b_->1$ and $b_+>0$, then the statement follows from Proposition \ref{prop:subnets} and Theorem \ref{thm:stability1}.
Otherwise the network has an embedded network of one of the types in Lemma \ref{prop:stability4}. Hence the statement follows from Propositions \ref{prop:embedded} and \ref{propn:AtMost3Case1Sec4}.\qedhere

\end{proof}

\appendix
\section{More on Gale dual systems}\label{app:galedual}

Let $\mathcal{V}$ be the zero dimensional subscheme of $(\RR_{>0})^{n}$ defined by the system of $n$ Laurent polynomials 
\begin{equation}\label{eq:compInt}
 Cx^W= \begin{bmatrix}
f_1(x_1,\dots,x_{n}) \\
\vdots\\
f_n(x_1,\dots,x_{n})
\end{bmatrix}  =0,
\end{equation} with $W$ and $C$ having Gale dual matrices $Q$ and $D$, respectively, chosen as in \S\ref{sec:galedual}. Recall that $W$ and $C$ are $n\times (n+l+1)$ matrices while $Q$ and $D$ are $(n+l+1)\times (l+1)$ matrices. 
We can now define a homomorphism of algebraic groups specifed by the monomial map determined by the exponents of \eqref{eq:compInt}\begin{align*}
&\varphi_{W} \colon (\RR_{>0})^{n} \to  (\RR_{>0})^{n+l} \times \{1\}  \subset \PP_{\RR}^{n+l}\\
&\varphi_{W}\colon x\mapsto x^W= \left[ x^{w_1}:\dots:x^{w_{n+l}}:1\right]^T.
\end{align*} The homomorphism $\varphi_{W}$ is dual to the homomorphism of free abelian groups $\iota_{W}\colon \ZZ^{l+n}\to \ZZ^n$ which maps the standard $i^{th}$ basis vector of $\ZZ^{l+n}$ to the column vector $w_i$. Let $[z_1:\cdots:z_{n+l+1}]$ be coordinates for $\PP_{\RR}^{n+l}$, then the polynomials $f_i$ are the pullbacks of linear forms $\Lambda_i$ under the monomial map $\varphi_{W} $, that is\begin{equation}
f_i=\varphi_{W}^*(\Lambda_i),\qquad \textrm{where}\quad \Lambda_i=\sum_{j=1}^{n+l+1} c_{i,j}z_j.\label{eq:Lambda_def}
\end{equation}
The scheme $\mathcal{V}$ defined by \eqref{eq:compInt} is the pullback of the linear space $L=V(\Lambda_1,\dots,\Lambda_n) \subset \PP_{\RR}^{n+l}$. Let $\ZZ W$ denote the integer lattice spanned by the columns of $W$. Since $\ZZ W=\ZZ^{n}$ we have that the intersection $Y=L\cap \varphi_{W}((\RR_{>0})^{n} )$ is proper (that is the intersection has the expected dimension) and the map $\varphi_{W}$ defines a scheme theoretic isomorphism between $\mathcal{V}$ and $Y$, see, for example \cite[Proposition~1.1]{bihan2008gale}.

Define the map $\psi_{\mathcal{V}}\colon \RR^l\to \PP_{\RR}^{l+n}$ given by $
 \psi_{\mathcal{V}}(y) = [d_1(y): \cdots : d_{n+l}(y):1]
$ where the $d_i(y)$ are the linear forms in $\R[y_1,\dots,y_{l}]$ defined by the rows of $D\cdot [y_1,\dots,y_{l},1]^T$ as in \eqref{eq:di}. Note that by construction $\psi_{\mathcal{V}}$ is an isomorphism from $\RR^l$ to the linear subspace $L$ of $\PP^{n+l}$. Hence we have an isomorphism of schemes given by $\psi_{\mathcal{V}}^{-1}\circ \varphi_W $ so that $
\mathcal{V}\cong Z=\psi_{\mathcal{V}}^{-1}(\varphi_W(\mathcal{V})) \subset \RR^l.$ The resulting isomorphic scheme $Z$ is referred to as the \textit{Gale dual scheme} of $\mathcal{V}$. 

We now give the equations which describe $Z$. 
Every integer linear relation, $ \sum_i \beta^{(i)} w_i=0$ with $\beta^{(i)}\in \ZZ$, among column vectors $w_i$ in $W$, corresponds to the Laurent monomial equality $
\prod_{i=1}^{n+l} z_i^{\beta^{(i)}}=1
$ on $\varphi_{W}((\RR_{>0})^{n} ) \subset \PP_{\RR}^{l+n}$, here we have that $z_i>0$. Pulling this relation back under the map $\psi_{\mathcal{V}}$ gives the relation \begin{equation}
\prod_{i=1}^{n+l} d_i(y)^{\beta^{(i)}}=1 \;\;\;{\rm in} \;\; \RR[y_1,\dots, y_l]\;\;\; {\rm with}\; \; d_i(y)>0.\label{eq:pullback_mon_eq}
\end{equation}There is one such relation \eqref{eq:pullback_mon_eq} for each row of $Q$ giving the following system of $l$ equations in $\RR[y_1,\dots,y_l]$ \begin{equation} \label{eq:GaleDualSys}
\prod_{i=1}^{n+l} d_i(y)^{q_{i,1}} = 1 , \dots, \prod_{i=1}^{n+l} d_i(y)^{q_{i,l}} = 1 ,\;\;\; {\rm such\; that}\; \; d_i(y)>0.
\end{equation} 
By construction $Z\subset (\RR_{>0})^l$ is the set of solutions to the system of equations \eqref{eq:GaleDualSys} and $Z\cong \mathcal{V}$ as schemes. The system \eqref{eq:GaleDualSys} is referred to as the \textit{Gale dual system} of the original system \eqref{eq:compInt}. Hence the one-to-one correspondence in Proposition \ref{Prop:GaleDual2} is an isomorphism of schemes.

\bigskip \bigskip

\begin{small}

\noindent
{\bf Acknowledgements.}
This work was partially funded by the Independent Research Fund of Denmark.
\end{small}
\bigskip 
\scriptsize

\bibliography{ref}
\bibliographystyle{alpha}
\end{document}